\documentclass[12pt,a4paper]{amsart}
\usepackage{amsfonts}
\usepackage{amscd}
\usepackage[utf8]{inputenc}
\usepackage{t1enc}
\usepackage[mathscr]{eucal}
\usepackage{indentfirst}
\usepackage{graphicx}
\usepackage{graphics}
\usepackage{pict2e}
\usepackage{epic}
\numberwithin{equation}{section}
\usepackage[margin=2.9cm]{geometry}
\usepackage{epstopdf} 
\usepackage{amsmath}
\allowdisplaybreaks
\usepackage{amssymb}
\usepackage{amsthm}
\usepackage{bbm}
\usepackage{bm}
\usepackage[dvipsnames]{xcolor}
\usepackage{esint}
\usepackage{tikz}
\usepackage{pgffor}
\usepackage{pgfplots}
\pgfplotsset{compat=1.16}
\usepackage{paralist}
\usepackage[onehalfspacing]{setspace}
\usepackage{amsfonts}
\usepackage{color}
\usepackage{nicefrac}
\newtheorem{thm}{Theorem}[section]
\usepackage[toc,page]{appendix}
\newtheorem{lemma}[thm]{Lemma}
\newtheorem{prop}[thm]{Proposition}
\newtheorem{defi}[thm]{Definition}
\newtheorem{coro}[thm]{Corollary}
\newtheorem{quest}[thm]{Question}
\newtheorem*{thm*}{Theorem}
\numberwithin{equation}{section}
\theoremstyle{remark}
\newtheorem{rem}[thm]{Remark}
\newtheorem{ex}[thm]{Example}
\usepackage{hyperref}
\newcommand{\A}{\mathcal{A}}
\newcommand{\B}{\mathcal{B}}
\newcommand{\Aop}{\mathbb{A}}

\newcommand{\V}{\vert}
\newcommand{\norm}{\Vert}
\newcommand{\R}{\mathbb{R}}
\newcommand{\N}{\mathbb{N}}
\newcommand{\Z}{\mathbb{Z}}

\renewcommand{\P}{\mathcal{P}}

\newcommand{\M}{\mathcal{M}}

\newcommand{\QA}{\mathcal{Q}_{\mathcal{A}}}

\newcommand{\Sphere}{\mathbb{S}}
\newcommand{\back}{\backslash}

\newcommand{\dx}{\,\textup{d}x}
\newcommand{\dt}{\,\textup{d}t}
\newcommand{\dy}{\,\textup{d}y}
\newcommand{\dz}{\,\textup{d}z}

\newcommand{\dH}{\,\textup{d}\mathcal{H}}
\newcommand{\Haus}{\mathcal{H}}
\newcommand{\dmu}{\,\textup{d}\mu}

\renewcommand{\phi}{\varphi}
\newcommand{\li}{\langle}
\newcommand{\re}{\rangle}

\everymath{\displaystyle}
\newcommand{\weakto}{\rightharpoonup}
\newcommand{\weakstar}{\overset{\ast}{\rightharpoonup}}

\DeclareMathOperator{\Lin}{Lin}

\DeclareMathOperator{\curl}{curl}
\DeclareMathOperator{\dist}{dist}

\DeclareMathOperator{\divergence}{div}

\DeclareMathOperator{\sgn}{sgn}

\DeclareMathOperator{\Sim}{Sim}
\DeclareMathOperator{\sym}{sym}
\DeclareMathOperator{\spt}{spt}

\DeclareMathOperator{\loc}{loc}

\DeclareMathOperator{\id}{id}
\begin{document}
\title{$L^{\infty}$-truncation of closed differential forms}
\author[Schiffer]{Stefan Schiffer}
\address{Insitute for Applied Mathematics, University of Bonn, Endenicher Allee 60, 53115 Bonn, Germany} 
\date{15.02.2021}
\email{schiffer@iam.uni-bonn.de}

\begin{abstract}
In this paper, we prove that for each closed differential form $u \in L^1(\R^N,(\R^N)^{\ast} \wedge ... \wedge (\R^N)^{\ast})$, which is almost in $L^{\infty}$ in the sense that
\begin{equation*}
	\int_{\{y \in \R^N \colon \V u(y) \V \geq L \}} \V u(y) \V \dy< \varepsilon
\end{equation*}
for some $L>0$ and a small $\varepsilon >0$, we may find a closed differential form $v$, such that $\norm u - v \norm_{L^1}$ is again small, and $v$ is, in addition, in $L^{\infty}$ with a bound  on its $L^{\infty}$ norm depending only on $N$ and $L$. In particular, the set $\{ v \neq u\}$ has measure at most $C \varepsilon.$ We then look at applications of this theorem. We are able to prove that the $\A$-$p$-quasiconvex hull of a set does not depend on $p$. Furthermore, we can prove a classification theorem for $\A$-$\infty$-Young measures.
\end{abstract}
\maketitle

\section{Introduction} 
A basic question in the calculus of variations and real analysis is the following: Consider a linear differential operator $\A \colon C^{\infty}(\R^N,\R^d) \to C^{\infty}(\R^N,\R^l)$ of first order with constant coefficients, and a bounded sequence of functions $u_n \in L^1(\R^N,\R^d)$ which satisfy $\A u_n =0$ in the sense of distributions and are close to a bounded set in $L^{\infty}$, i.e. 
\begin{equation} \label{eq:almost_linfty} \lim_{n \to \infty} \int_{\{x \in \R^N \colon \V u_n(x) \V \geq L\}} \V u_n \V \, \dx=0
\end{equation}
for some $L>0$. Does there exist a sequence of functions $v_n$, such that $\A v_n =0$, $\norm v_n \norm_{L^{\infty}} \leq CL $ and $(u_n - v_n) \to 0$ in measure (in $L^1$)?

This question was answered first by \textsc{Zhang} in \cite{Zhang} for sequences of gradients ($u_n = \nabla w_n$), i.e. for the operator $\A= \curl$, which assigns to a function $u \colon \R^N \to \R^N$ the skew-symmetric $(N \times N)$-matrix with entries $\partial_i u_j - \partial_j u_i$. \textsc{Zhang's} proof, which builds on the works of \textsc{Liu} \cite{Liu} and \textsc{Acerbi-Fusco} \cite{Acerbi}, proceeds as follows. Denote by $M f$ the Hardy-Littlewood maximal function of $f \in L^1_{\loc}(\R^N,\R^d)$ and let $u_n = \nabla w_n$. The estimate \eqref{eq:almost_linfty} implies that the sets $X^n= \{ M( \nabla w_n)\geq L'\}$ have small measure for large $n$. One then uses (c.f. \cite{Acerbi}) that \begin{equation} \label{eq:AFestimate}
 \V w_n(x) - w_n(y) \V \leq C L' \V x -y \V ',\quad x,y \in \R^N \back X^n,
\end{equation}
i.e. $w_n$ is Lipschitz continuous on $\R^N \back X^n$. The fact that Lipschitz continuous functions on closed subsets of $\R^N$ can be extended to Lipschitz continuous functions on $\R^N$ with the same Lipschitz constant \cite{KB} yields the result.

In this paper, we show that the answer to the previously formulated question is also positive for sequences of differential forms and $\A=d$, the operator of exterior differentiation.

Let us denote by $\Lambda^r$ the r-fold wedge product of the dual space $(\R^N)^{\ast}$ of $\R^N$ and by $d \colon C^{\infty}(\R^N,\Lambda^r) \to C^{\infty}(\R^N,\Lambda^{r+1})$ the exterior derivative w.r.t. the standard Euclidean geometry on $\R^N$.

\begin{thm}[$L^{\infty}$-truncation of differential forms] \label{Thm:intro}
Suppose that we have a sequence $u_n \in L^1(\R^N,\Lambda^r)$ with $d u_n=0$ (in the sense of distributions), and that there exists an $L>0$ such that \begin{displaymath}
\int_{\{y \in \R^N \colon \V u_n(y) \V >L \}} \V u_n(y) \V \dy \longrightarrow 0 \quad \text{as } n \to \infty.
\end{displaymath}
There exists a constant $C_1=C_1(N,r)$ and a sequence $v_n \in L^{\infty}(\R^N,\Lambda^r)$ with $dv_n=0$ and \begin{enumerate} [i)]
\item $\norm v_n \norm_{L^{\infty}(\R^N,\Lambda^r)} \leq C_1 L$;
\item $\norm v_n - u_n \norm_{L^1(\R^N,\Lambda^r)} \to 0$ as $n \to \infty$;
\item $\V \{ y \in \R^N \colon v_n(y) \neq u_n(y) \} \V \to 0$.
\end{enumerate}
\end{thm}

An analogous version of Theorem \ref{Thm:intro} holds if $\R^N$ is replaced by the $N$-torus $T_N$ (c.f. Theorem \ref{main}) or by an open Lipschitz set $\Omega$ and functions $u$ with zero boundary data (c.f. Propostion \ref{main4}). Moreover, the result immediately extends to $\R^m$-valued forms by taking truncations coordinatewise (c.f. Proposition \ref{main5}).

In particular, the result of Theorem \ref{Thm:intro} includes a positive answer to the question previously raised for the differential operator $\A =\divergence$ after suitable identifications of $\Lambda^{N-1}$ and $\Lambda^N$ with $\R^N$ and $\R$, respectively.

One key ingredient in the proofs is a version of the Acerbi-Fusco estimate \eqref{eq:AFestimate} for simplices rather than pairs of points. For the estimate, let us consider $\omega \in C_c^1(\R^N,\Lambda^r)$ with $d\omega=0$ and let $D$ be a simplex with vertices $x_1,...,x_{r+1}$ and a normal vector $\nu^r \in \R^N \wedge... \wedge \R^N$ (c.f. Section \ref{SecStokes} for the precise definition). Assume that $M\omega (x_i) \leq L$
for $i=1,...,r+1$. Then \begin{equation} \label{Lipschitzversion}
\left \V  \int_D \omega ( \nu^r) \right  \V \leq C(N) L \sup_{1 \leq i,j \leq r+1} \V x_i -x_j \V ^r.
\end{equation}
The second ingredient is a geometric version of the Whitney extension theorem, which may be of independent interest.

Combining \eqref{Lipschitzversion} and the extension theorem, one easily obtains the assertion for smooth closed forms. The general case follows by a standard approximation argument.

Truncation results like the result by \textsc{Zhang} or Theorem \ref{Thm:intro} have immediate applications in the calculus of variations. In particular, they provide characterisations of the $\A$-quasiconvex hulls of sets (c.f. Section \ref{Aqchulls}) and the set of Young-measures generated by sequences satisfying $\A u_n=0$. The classical result for sequences of gradients (i.e. sequences of functions $u_n$ satisfying $\curl u_n=0$) goes back to \textsc{Kinderlehrer} and \textsc{Pedregal} \cite{Kinderlehrer,Pedregal}.
Here, we show the natural counterpart of their characterisation result, whenever the operator $\A$ admits the following $L^{\infty}$-truncation result:

We say that $\A$ satisfies the property (ZL) if for all sequences $u_n \in L^1(T_N,\R^d) \cap \ker \A$, such that there exists an $L>0$ with 
\begin{displaymath}
\int_{\{y \in T_N \colon \V u_n(y) \V >L \}} \V u_n(y) \V \dy \longrightarrow 0 \quad  \text{as } n \to \infty,
\end{displaymath}
there exists a $C=C(\A)$ and a sequence $v_n \in L^1(T_N,\R^d) \cap \ker \A$ such that \begin{enumerate} [i)]
\item $\norm v_n \norm_{L^{\infty}(T_N,\R^d)} \leq C L$;
\item $\norm v_n - u_n \norm_{L^1(T_N,\R^d)} \to 0$ as $n \to \infty$.
\end{enumerate}

By \textsc{Zhang} \cite{Zhang}, the property (ZL) holds for $\A=\curl$ and a version of Theorem \ref{Thm:intro} shows this for $\A=d$ (Corollary \ref{main2}). Further examples are shortly discussed in Example \ref{ex:op}.

For the characterisation of Young measures, recall that $\spt \nu$ denotes the support of a (signed) Radon measure $\nu \in \mathcal{M}(\R^d)$, and  for $f \in C(\R^d)$ \begin{displaymath}
\li \nu, f \re :=\int_{\R^d} f \textup{d}\mu.
\end{displaymath}
If the property (ZL) holds for some differential operator $\A$, then one is able to prove the following statement.
\begin{thm}[Classification of $\A$-$\infty$-Young measures] Let $\A$ satisfy (ZL).
A weak$*$ measurable map $\nu: T_N \to \M(\R^d)$ is an $\A$-$\infty$-Young measure if and only if $\nu_x \geq 0$ a.e. and there exists $K  \subset \R^d$ compact and $u \in L^{\infty}(T_N,\R^d) \cap \ker \A$ with \begin{enumerate} [i)]
\item $\spt \nu_x \subset K$ for a.e. $x \in T_N$;
\item $ \li \nu_x, id \re = u$ for a.e. $x \in T_N$;
\item $\li \nu_x, f \re \geq f(\li \nu_x, id \re)$ for a.e. $x \in T_N$ and  all continuous and $\A$-quasiconvex $f:\R^d \to \R$.
\end{enumerate}
\end{thm}

We close the introduction with a brief outline of the paper. In Section \ref{secaux}, we introduce some notation, recall some basic facts from multilinear algebra and the theory of differential forms and prove estimate \eqref{Lipschitzversion}. Section \ref{secWhitney} is devoted to the proof of the geometric Whitney extension theorem. In Section \ref{Secmain}, the proof of the truncation result (and its local and periodic variant) is given. Section \ref{secappl} discusses the applications to $\A$-quasiconvex hulls and $\A$-Young measures. 
Many of the arguments here follow the arguments in \cite{Kinderlehrer}. The necessary adaptations in our setting are discussed in an Appendix.


\section{Preliminary results} \label{secaux}
\subsection{Notation} 
We consider  an open and bounded Lipschitz set $\Omega \subset \R^N$ and denote by $T_N$ the $N$-dimensional torus, which arises from identifying faces of $[0,1]^N$. We may identify functions $f \colon T_N \to \R^d$ with $\Z^N$-periodic functions $\tilde{f} \colon \R^N \to \R^d$, and vice versa. We write $B_\rho(x)$ to denote the ball with radius $\rho$ and centre $x$. Denote by $\mathcal{L}^N$ the Lebesgue measure and, for a set $X \subset \R^N$, \begin{displaymath}
\V X \V := \mathcal{L}^N(X).
\end{displaymath}
For a measure $\mu$ on $\R^N$ and a $\mu$-measurable set $A \subset \R^N$ with $0<\mu(A)<\infty$ define the average integral of a $\mu$-measurable function f via 
\begin{displaymath}
\fint_A f \dmu = \frac{1}{\mu(A)} \int_A f \dmu.
\end{displaymath}
For $k\in \N$ write $[k] = \{1,...,k\}$. For a normed vector space $V$ we denote by $V^{\ast}$ the dual space of $V$.

Define the space $\Lambda^r$ as the $r$-fold wedge product of $(\R^N)^{\ast}$, i.e. \begin{displaymath}
\Lambda^r = \underbrace{(\R^N) ^{\ast} \wedge ... \wedge (\R^N)^{\ast}}_ {\substack{r\text{ copies}}}
\end{displaymath}
and similarly the space $\Lambda_r$ as the $r$-fold wedge product of $\R^N$.
Then $\Lambda^r$ and $\Lambda_r$ are finite-dimensional vector spaces. For $\R^N$ denote by $\{e_i \}_{i \in [N] }$ the standard basis and by $\cdot$ the standard scalar product. For $(\R^N)^{\ast}$ denote by $\theta_1,...,\theta_N$ the corresponding dual basis of $(\R^N)^{\ast}$, i.e. $\theta_i$ is the map $y \mapsto y \cdot e_i$. 

For $k \in I_r := \{ l \in [N]^r \colon l_1 < l_2 <...<l_r\}$ the vectors \begin{equation} \label{dxkr}
 e^{k,r} = e_{k_1} \wedge e_{k_2} \wedge... \wedge e_{k_r}
\end{equation}
form a basis of $\Lambda_r$. Denote by $\cdot^r$ the scalar product with respect to this basis, i.e. for $k,l \in I_r$ \begin{displaymath}
e^{k,r} \cdot^r e^{l,r} = \left \{ \begin{array}{ll} 1 & k=l, \\ 0 & k\neq l.\end{array} \right.
\end{displaymath}
    This also provides us with a suitable norm on $\Lambda_r$, which we denote by $\norm \cdot \norm_{\Lambda_r}$. Similarly, using the standard basis of $(\R^n)^{\ast}$, we define a basis $\theta^{k,r}$ and a norm $\norm \cdot \norm_{\Lambda^r}$. Also note that for $0 \leq s \leq r$ there exists (up to sign) a natural map $\Lambda^r \times \Lambda_s \mapsto \Lambda^{r-s}$, as $\Lambda^s$ is the dual space of $V^s$ and $\Lambda^r= \Lambda^s \wedge \Lambda^{r-s}$. In particular, in the special case $s=1$ for $h_1,...,h_r \in \R^{N\ast}$ and $y \in \R^N$  \begin{equation} \label{s1}
(h_1 \wedge .... \wedge h_r)(y) = \sum_{i=1}^r (-1)^{i-1} h_i(y) h_1 \wedge ... \wedge h_{i-1} \wedge h_{i+1} ... \wedge h_r.
\end{equation}
In the case $s=r$ and for $h_1,...,h_r \in (\R^{N})^{\ast}$ and $y_1,...,y_r \in \R^N$ \begin{equation}\label{sr}
(h_1 \wedge .... \wedge h_r)(y_1 \wedge ... \wedge y_r) = \sum_{ \sigma \in S_r} \left( \sgn(\sigma) \prod_{i=1}^r h_i(y_{\sigma(i)})\right),
\end{equation}
where $S_r$ denotes the group of permutations of $\{1,...,r\}$. \eqref{sr} also gives us a representation of the map $\Lambda^r \times \Lambda_s \mapsto \Lambda^{r-s}$ as for $h \in \Lambda^r$, $x \in \Lambda_s$ we may consider the element of $\Lambda^{r-s}=(\Lambda_{r-s})^{\ast}$ defined by \begin{displaymath}
 z \longmapsto h ( x \wedge z), \quad z \in \Lambda_{r-s}.
\end{displaymath}
Let us shortly remark that this notation is slightly different to the usual notation for interior products.

Moreover, note that the space $\Lambda^N$ is isomorphic to $\R$ via the map $I^N$ defined by \begin{displaymath}
 a ~ \theta_1 \wedge ... \wedge \theta_N \longmapsto a \in \R.
\end{displaymath}

\subsection{Differential forms}
In the following, we will define all objects for an open set $\Omega \subset \R^N$, but these definitions are also valid for $\R^N$ and $T_N$ respectively.

We call a map $f \in L^1_{\loc}(\Omega,\Lambda^r)$ an \textbf{$r$-differential form} on $\Omega$. We define the space \begin{displaymath}
\Gamma = \bigcup_{r \in \N} C^{\infty}(\Omega,\Lambda^r).
\end{displaymath}

It is well-known (c.f \cite{Cartan,diffgeo}) that there exists a linear map $d\colon \Gamma \mapsto \Gamma$, called the \textbf{exterior derivative}  with the following properties \begin{enumerate}[i)]
\item $d^2 = d \circ d = 0$,
\item $d$ maps $C^{\infty}(\Omega,\Lambda^r)$ into $C^{\infty}(\Omega,\Lambda^{r+1})$,
\item We have the \textbf{Leibniz rule}: If $\alpha \in C^{\infty}(\Omega,\Lambda^r)$ and $\beta \in C^{\infty}(\Omega,\Lambda^s)$, then \begin{equation} \label{Leibniz}
d (\alpha \wedge \beta) = d \alpha \wedge \beta + (-1)^r \alpha \wedge \beta,
\end{equation}
\item $d: C^{\infty}(\Omega,\Lambda^0) \to C^{\infty}(\Omega,\Lambda^1)$ is the gradient via the identification $\Lambda^0 =\R$, $\Lambda^1 =  (\R^N)^{\ast} \cong \R^N$.
\end{enumerate}
This map $d$ has the following representation in terms of the standard coordinates (c.f. \cite{diffgeo}). Let $\omega \in C^{\infty}(\Omega,\Lambda^r)$, which, for some $a_k \in C^{\infty}(\Omega,\R)$, can be written as \begin{displaymath}
 \omega(y) = \sum_{k \in I_r} a_k(y) \theta^{k,r}.
\end{displaymath}
 Then \begin{equation} \label{coordinate}
d \omega(y) = \sum_{k \in I_r} \sum_{l \in [N]} \partial_l a_k(y) \theta_l \wedge \theta^{k,r}.
\end{equation}

\begin{rem} For a fixed $r \in \{0,...,N-1\}$ we can identify $d \colon C^{\infty}(\Omega, \Lambda^r) \mapsto C^{\infty}(\Omega, \Lambda^{r+1})$ with a differential operator $\A$. By definition, for $r=0$, $d$ can be identified with the gradient. For $r=1$, after a suitable identification of $\Lambda^2$ with $\R^{N \times N}_{skew}$, $d= \curl$, which is the differential operator mapping $u \in C^{\infty}(\Omega,\R^N)$ to $\curl u \in C^{\infty}(\Omega,\R^{N \times N}_{skew})$ defined by \begin{displaymath}
(\curl u )_{lk} = \partial_l u_k - \partial_k u_l.
\end{displaymath}
If $r=N-1$, after identifying $\Lambda^{N-1}$ with $\R^N$ and $\Lambda^N$ with $\R$, the differential operator $d$ becomes the divergence of a vector field which is defined for $u \in C^{\infty}(\Omega,\R^N)$ by \begin{displaymath}
\divergence u = \sum_{k=1}^N \partial_k u_k.
\end{displaymath}
\end{rem}

\begin{lemma} \label{prodrule} We have the following product rules for $d$: \begin{enumerate} [i)]
    \item Let $\omega \in C^1(\Omega,\Lambda^1)$, $z \in \R^N= \Lambda_1$. Then \begin{equation} \label{prod1}
    d (\omega(\cdot)(\cdot-z)) = \nabla \omega(\cdot) (\cdot-z) + \omega(\cdot),
    \end{equation}
    where we define $\nabla \omega(\cdot) (\cdot-z) \in C(\Omega,\Lambda^1)$ as follows: \\If $\omega = \sum_{i=1}^N \omega_i \theta_i$ and $(y-z) = \sum_{i=1}^N (y-z)_i e_i$, then \begin{displaymath}
    \nabla \omega(y-z) := \sum_{l=1}^N \sum_{i=1}^N \partial_l \omega_i (y-z)_i \theta_l.
    \end{displaymath}
    \item There is a linear bounded map $D^{1,r} \in \Lin((\Lambda^r \times \R^N)\times \R^N, \Lambda^r)$ such that for  $\omega \in C^1(\Omega,\Lambda^r)$, $z \in \R^N$ we have \begin{equation}
        d (\omega(\cdot)(\cdot-z)) = D^{1,r}( \nabla \omega,(\cdot-z)) + \omega(\cdot).
    \end{equation}
    \item There is a linear and bounded map $D^{s,r} \in \Lin((\Lambda^r \times\R^N) \times \Lambda_s, \Lambda^{r-s})$ such that for $\omega \in C^1(\Omega, \Lambda^r)$, $z \in \R^N$, $z_2 \in \Lambda_{s-1}$ \begin{equation} \label{productrule}
    d ( \omega(\cdot) ((\cdot-z)\wedge z_2)) = D^{s,r} (\nabla \omega, (\cdot-z)\wedge z_2) + (-1)^{s-1}\omega(\cdot) (z_2).
    \end{equation}
\end{enumerate}
\end{lemma}
\begin{proof}
i) simply follows from a calculation, i.e., if as mentioned \begin{displaymath} 
\omega(y) = \sum_{i=1}^N \omega_i(y) \theta_i \quad \text{ and }(y-z) = \sum_{i=1}^N (y-z)_i e_i, \end{displaymath} then \begin{align*}
d(\omega(y)(y-z))& = \sum_{l=1}^N \partial_l(\omega(y)(y-z)) \theta_l \\
            &= \sum_{i,l=1}^N \partial_l \omega_i(y) (y-z)_i \theta_l+ \sum_{l=1}^N \omega_l(y) \theta_l,
\end{align*}
which is what we claimed. ii) then follows from i) and using \eqref{s1}. Likewise, iii) then follows from ii).
\end{proof}
\begin{defi}
For $\omega \in L^1_{loc}(\Omega,\Lambda^r)$ and $u \in L^1_{\loc}(\Omega, \Lambda^{r+1})$ we say that $d \omega = u$ in the sense of distributions  if for all $\phi \in C_c^{\infty}(\Omega, \Lambda^{N-r-1})$ we have \begin{displaymath}
\int_{\Omega} d \phi \wedge \omega   = (-1)^{N-r} \int_{\Omega } \phi \wedge u .
\end{displaymath}
\end{defi}
Note that this definition is equivalent to the following formula: For all $\phi \in C_c^{\infty}(\Omega,\Lambda^s) $ with $0 \leq s \leq  N-r-1$  \begin{displaymath}
(-1)^{r+1} \int_{\Omega} \omega \wedge d\phi= - \int_{\Omega} u \wedge \phi.
\end{displaymath}

\subsection{Stoke's theorem on simplices} \label{SecStokes}
We want to establish a suitable notion of Stoke's theorem for differential forms on simplices. Let $1 \leq r \leq N$ and  $x_1,...,x_{r+1} \in \R^N$. Define the simplex $\Sim(x_1,...,x_{r+1})$ as the convex hull of $x_1,...,x_{r+1}$. We call this simplex degenerate, if its dimension is strictly less than $r$. 

For $i \in \{1,...,r+1\}$ consider $\Sim(x_1,...x_{i-1},x_{i+1},...,x_{r+1}) =: \Sim^i(x_1,...x_{r+1})$. This is an $(r-1)$ dimensional face of $\Sim(x_1,...,x_{r+1})$ and a subset of the boundary of the manifold $\Sim(x_1,...,x_{r+1})$, which, for simplicity, will be  denoted by $\partial \Sim(x_1,...,x_{r+1})$. Suppose first that we are given the simplex \begin{displaymath}
\{\lambda \in [0,1]^r \ \colon \sum_{i=1}^r \lambda_i \leq 1\} \times \{0\}^{N-r} = \Sim(0,e_1,...,e_r) \subset \R^r \times \{0\}^{N-r} \subset \R^N.
\end{displaymath}
Then the classical version of Stoke's theorem on oriented manifolds reads that for every differential form $\tilde{\omega} \in C^1(\R^r \times \{0\}^{N-r}, \R^r \wedge ... \wedge \R^r)$ we have \begin{equation} \label{Stokes1}
\int_{\Sim(0,e_1,...,e_r)} d\tilde{\omega}(y) \dH^{r}(y) = \int_{\partial^{\ast} \Sim(0,e_1,...,e_r)} \tilde{\omega}(y) \wedge \nu(y) \dH^{r-1}(y).
\end{equation}
In \eqref{Stokes1}, $\nu(y)$ denotes the outer normal unit vector at $y\in \partial^{\ast} \Sim(0,e_1,...e_r)$ and $\partial^{\ast}$ is the reduced boundary of the simplex, where this outer normal exists (the interior of all $(r-1)$-dimensional faces). In our case, we are given a differential form with the underlying space being $\R^N$ and not $\R^r$ (the tangential space of the manifold/simplex), hence we can modify \eqref{Stokes1} to get for $\omega \in C^1(\R^N,\Lambda^r)$
\begin{align} \label{Stokes15} 
\int_{\Sim(0,e_1,...,e_r)}& d\tilde{\omega}(y) (e_1 \wedge ... \wedge e_r) \dH^r(y) \nonumber \\ &= \sum_{i=1}^r (-1)^i \int_{\Sim(0,...,e_{i-1},e_{i+1},...,e_r)} \omega(y) (e_1 \wedge ... \wedge e_{i-1} \wedge e_{i+1} \wedge ...\wedge e_r) \\
&+ \int_{\Sim(e_1,...,e_r)}  2^{-r/2}\omega(y)((e_2-e_1) \wedge (e_3-e_2) \wedge ... \wedge (e_{r}-e_{r-1})). \nonumber
\end{align}

Let us write for simplicity that for $x_1,...,x_{r+1} \in \R^N$ \begin{displaymath}
\nu^r(x_1,...,x_{r+1}) = ((x_2-x_1) \wedge (x_3 -x_2) \wedge  ... \wedge (x_{r+1}-x_r)) \in V_r.
\end{displaymath}
The map $\nu^r$ has the following properties:
\begin{enumerate} [i)]
\item $\nu^r$ is alternating, i.e. for a permutation $\sigma \in S_r$: \begin{displaymath}
\nu^r(y_1,...,y_{r+1}) = \sgn(\sigma) \nu^r(y_{\sigma(1)},...,y_{\sigma(r+1)}).
\end{displaymath}
 \item We have the relation \begin{displaymath}
 \norm \nu^r(y_1,...,y_{r+1}) \norm_{\Lambda_r} = r \Haus ^r(\Sim(y_1,...,y_{r+1})) .
 \end{displaymath}
\end{enumerate}
A linear change of coordinates from $\Sim(0,e_1,..,e_r)$ to $\Sim(x_1,...,x_{r+1})$ leads from \eqref{Stokes15} to the following: For $\omega \in C^{\infty}(\R^N,\Lambda^{r-1})$ and $x_1,...x_{r+1} \in \R^N$  \begin{align} \label{Stokes2}
&\frac{1}{r} \fint_{\Sim(x_1,...,x_{r+1})} d \omega(y) (\nu^r(x_1,...,x_{r+1})) \dH ^r(y) \\ \nonumber
~&= \sum_{i=1}^{r+1} \frac{(-1)^i}{r-1} \fint_{\Sim^i(x_1,...x_{r+1})} \omega(y)( \nu^{r-1} (x_1,...,x_{i-1},x_{i+1},...x_{r+1})) \dH  ^{r-1}(y),
\end{align}

\subsection{The maximal function}
The Hardy-Littlewood maximal function for $u \in L^1_{\loc}(\R^N,\R^d)$ is defined by \begin{displaymath}
Mu(x) = \sup_{R>0} \fint_{B_R(x)} \V u(y) \V \dy.
\end{displaymath}
Again, we can also define the maximal function for functions on the torus using the identification with periodic functions.

\begin{prop}  [Properties of the maximal function] (c.f. \cite{Stein}) \label{maximalfct}
$M$ is sublinear, i.e. $ M(u+v)(y) \leq Mu(y) + Mv(y)$ for all $u,v \in L^1_{\loc}(\R^N,\R^d)$ and $y\in \R^N$. Moreover, $M: L^p(\R^N,\R^d) \to L^p(\R^N,\R)$ is bounded for $1 < p \leq \infty$ and bounded from $L^1$ to $L^{1,\infty}$. In particular, this means that for $1 \leq p < \infty$
\begin{displaymath}
\left \V \{ Mu > \lambda \} \right \V  \leq C_p \lambda ^{-p} \norm u \norm_{L^p(\R^N,\R^d)}^p.
\end{displaymath}
\end{prop}

If $u \in L^p_{\loc}(\R^N,\R^d)$ is a $\Z^N$-periodic function, i.e. $u \in L^p(T_N,\R^d)$, then \begin{displaymath}
\V \{M u >\lambda\} \cap [0,1]^N \V \leq C_p \lambda^{-p} \norm u \norm_{L^p([0,1]^N,\R^d)}^p.
\end{displaymath}

We now come to a key lemma for our main theorem. 

\begin{lemma} \label{maximalf}
There exists a constant $C = C(N,r)$ such that for all $\omega \in C^1(\R^N,\Lambda^{r})$, $\lambda >0$ with $d \omega=0$ and $x_1,...,x_{r+1} \in \{ M\omega \leq \lambda\}$ we have \begin{displaymath}
\left \V \fint_{\Sim(x_1,...,x_{r+1})}  \omega (\nu^r(x_1,...,x_{r+1}))  \right \V \leq C \lambda \max_{1\leq i,j \leq r+1} \V x_i -x_j \V ^r.
\end{displaymath}
\end{lemma}
This lemma is so to speak our version of Lipschitz continuity. In particular, it has been proven (for example in \cite{Acerbi}) that for $u \in W^{1,1}_{loc}(\R^N,\R^m)$ and for $y_1,y_2 \in\{M \nabla u(x) \leq L \}$ \begin{displaymath}
\left \V \int_0^1 \nabla u(ty_1+(1-t)y_2) \cdot (y_1 -y_2) \dt \right \V = \V u(y_1) -u(y_2) \V \leq C L \V y_1 -y_2 \V.
\end{displaymath}
Hence, one should view Lemma \ref{maximalf} as a generalisation of this result.
\begin{proof} For simplicity write $\V \omega \V := \norm \omega \norm_{\Lambda^r}$.

It suffices to show that there exists $z \in \R^N$ such that \begin{equation} \label{claim1}
\sum_{i=1}^{r+1} \int_{\Sim(x_1,...x_{i-1},z,x_{i+1},...)} \V \omega \V \dH^{r}(y) \leq C \lambda \max_{i,j \in [r+1]} \V x_i -x_j \V^r 
\end{equation}
using that \begin{align} \label{Stokesappl}
\sum_{i=1}^{r+1} \int_{\Sim(x_1,...x_{i-1},z,x_{i+1},...)} \omega( \nu^{r-1}(x_1,...x_{i-1},z,x_{i+1},...)) \dH^{r}(y) \\ \nonumber= \int_{\Sim(x_1,...,x_{r+1})} \omega (\nu^{r-1}(x_1,...,x_{r+1})) \dH^{r}(y).
\end{align}
This equation \eqref{Stokesappl} can be verified by Stoke's theorem \eqref{Stokes2}, using that boundary terms with a simplex with vertex $z$ cancel out on the left-hand side of \eqref{Stokesappl}. 
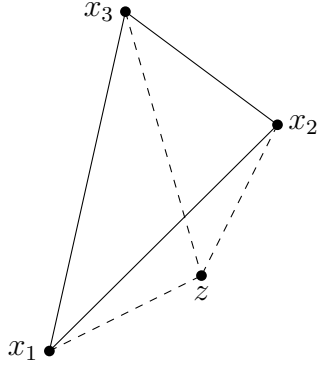
\begin{figure}
\centering
\begin{tikzpicture}[domain=0:6]
\coordinate[label=left:$x_1$] (A) at (1,1);
\coordinate[label=right:$x_2$] (B) at (4,4);
\coordinate[label=left:$x_3$] (C) at (2,5.5);
\coordinate[label=below:$z$] (D) at (3,2);
\fill (A) circle (2pt); 
\fill (B) circle (2pt);
\fill (C) circle (2pt);
\fill (D) circle (2pt); 
\draw (A) -- (B);
\draw (A) -- (C);
\draw (B) -- (C);
\draw[dashed] (A) -- (D);
\draw[dashed] (B) -- (D);
\draw[dashed] (C) -- (D);
\end{tikzpicture} 
\caption{Illustration of \eqref{Stokesappl} for $r=2$. The integrals on the dashed $1$-dimensional faces cancel out in \eqref{Stokesappl} after applying Stoke's theorem.} \label{picture1} \end{figure}
We now prove \eqref{claim1}. W.l.o.g. $R= \max_{i,j \in[r+1]} \V x_i-x_j \V = \V x_1-x_2 \V$. Note that there exists a dimensional constant $C_1$ such that \begin{displaymath}
 \V B_R(x_1) \cap B_R(x_2) \V \leq C_1 R^N.
\end{displaymath}
First, consider $x_1,...,x_r \in B_R(x_1)$. For $z \in B_R(x_1)$ define $E(z)$ to be the $r$-dimensional hyperplane going through $x_1,...,x_r$ and $z$. This is well-defined if $z$ is not in the $(r-1)$ dimensional hyperplane $F$ going through $x_1,...,x_r$. Note that for $z,\tilde{z} \notin F$ \begin{displaymath} z \in E(\tilde{z}) \Leftrightarrow \tilde{z} \in E(z). \end{displaymath}
\noindent As $M  \omega  (x_1) \leq \lambda$, we know that \begin{displaymath}
\int_{B_R(x_1)} \V \omega \V (z) \dz \leq \lambda b_N R^N,
\end{displaymath}
where $b_N$ is the volume of the $N$-dimensional unit ball $B_1(0)$. As $\Haus ^r(E(z) \cap B_R(x_1)) = b_r R^r$, it also follows that 
\begin{displaymath}
\int_{B_R(x_1)} \int_ {E(z) \cap {B_R(x_1)}} \V \omega \V (y) \dH^{r}(y) \dz \leq \lambda b_N b_r  R^{N+r}.
\end{displaymath}
Using that $\Sim(x_1,...,x_r,z) \subset E(z) \cap B_R(x_1)$, we conclude that for $\mu >0$  \begin{equation} \label{measureest}
\left \V \left \{z \in B_R(x_1) \colon \left \V \int_{\Sim(x_1,...,x_r,z)} \V  \omega \V (y) \dy \right \V \geq \mu \right \} \right \V  \leq \frac{\lambda b_r b _N R^{N+r}} {\mu}.
\end{equation}
Choose now $\mu^{\ast}= 2(r+1) b_r b_N R^r\lambda C_1^{-1}$. Plugging this into \eqref{measureest}, we see that the measure of this set is smaller that $ R^N(2 (r+1))^{-1}$.
Repeating this procedure for all $(r-1)$-dimensional faces of $\Sim(x_1,...,x_{r+1})$, we get that for  $i>1$
\begin{displaymath}
\left \V \left \{z \in B_R(x_1) \colon \left \V \int_{\Sim(x_1,...,x_{i-1},z,x_{i+1},...)} \V \omega \V (y) \dH^{r}(y) \right \V \geq  \mu^{\ast} \right \} \right \V  \leq 
\frac{C_1 R^N}{2 (r+1)},
\end{displaymath}
and for $i=1$
\begin{displaymath}
\left \V \left \{z \in B_R(x_2) \colon \left \V \int_{\Sim(z,x_2,...x_{r+1})} \V  \omega \V (y) \dH^{r}(y) \right \V \geq  \mu^{\ast} \right \} \right \V  \leq 
\frac{C_1 R^N}{2 (r+1)}.
\end{displaymath}
Hence, there exists $z \in B_R(x_1) \cap B_R(x_2)$ such that all the integrals in the sum of \eqref{claim1}  are smaller than $((2(r+1))^{-1} b _r b _N C_1^{-1})  R^r\lambda$. This is what we wanted to prove.
\end{proof}

\section{A Whitney-type extension theorem} \label{secWhitney}
First, let us recall the following Lipschitz extension theorem.
\begin{thm}[Lipschitz extension theorem] 
Let $X \subset \R^N$ be a closed set and $u \in C(X,\R^d)$ such that \begin{equation} \label{LipschitzonX}
\V u(x) - u (y) \V \leq L \V x - y \V.
\end{equation}
Then there exists a function $v \in C(\R^N,\R^d)$ with $v_{\V X} =u$ and such that $v$ is Lipschitz on $\R^N$ with Lipschitz constant at most $C(N)L$ (i.e. the Lipschitz constant does not depend on $X$).
\end{thm}
Of course, there are several ways to prove such a theorem, even with $C(N)=1$ \cite{KB}. However, \textsc{Whitney's} proof \cite{Whitney} plays with the geometry of $\R^N$ quite nicely. A similar geometric ideas lies behind our proof for closed differential forms. First, let us define an analogue of \eqref{LipschitzonX}.

Suppose that $X$ is a closed subset of $\R^N$, such that $X^C = \R^N \backslash X$ is bounded and $\V \partial X \V  =0$. \\
Let $u \in C_c^{\infty}(\R^N,\Lambda^r)$  with $du =0$. Let $L>0$ be such that  $\norm u \norm_{L^{\infty}(X)} \leq L$ and that for all $x_1,...,x_{r+1} \in X$ we have  \begin{equation} \label{Lipschitzprop}
\left \V \fint_{\Sim(x_1,...,x_{r+1})}  u(y) (\nu^r(x_1,...,x_{r+1})) \dy \right \V \leq L \max \V x_i -x_j \V ^r.
\end{equation}

\begin{lemma}[Whitney-type extension theorem] \label{extension}
There exists a constant $C= C(N,r)$ such that for all $u \in C_c^{\infty}(\R^N,\Lambda^r)$ and $X$ meeting the requirements above there exists $v \in L^1_{\loc}(\R^N,\Lambda^r)$ with \begin{enumerate} [i)]
\item $dv =0$ in the sense of distributions;
\item $v (y) = u (y) $ for all $y \in X$;
\item $ \norm v \norm_{L^{\infty}} \leq  CL$.
\end{enumerate}
\end{lemma}
\begin{rem}
The constant $C$ does not depend on the choice of $u$ or $X$, it is only important that the pair $(u,X)$ satisfies \eqref{Lipschitzprop}. The assumption that $X^C$ is bounded can be dropped, the assumption $\V \partial X \V =0$ makes the proof much easier.
\end{rem}
\begin{rem}
As one can see in the proof, the assumption $u\in C_c^{\infty}(\R^N,\Lambda^r)$ can be weakened to $u \in C_c^1(\R^N,\Lambda^r)$, as we only need the first derivative of $u$. However, it is important to remember that we cannot prove Lemma \ref{extension} for the even weaker assumption $u \in L^1_{\loc}$, as \eqref{Lipschitzprop} is not well-defined.
\end{rem}
For the proof we follow the classical approach by Whitney with a few little twists. First, we will define  the extension in \eqref{vdef}. Then we prove that $v$ satisfies properties i)-iii). ii) and iii) are quite easy to see from the definition of $v$, however it is hard to verify that i) holds. On the one hand, we show that the strong derivative of $v$ exists almost everywhere, namely in $\R^N \back \partial X$ and that $dv=0$ almost everywhere. Here we need $ \V \partial X \V =0$. On the other hand, we then prove that the distributional derivative $dv$ is in fact also an $L^1$ function, yielding that $dv=0$ in the sense of distributions.

We now start with the definition of the extension. Let us recall (c.f. \cite{Stein}) that for $X \subset \R^N$ closed we can find a collection of pairwise disjoint open cubes $\{Q^{\ast}_i\}_{i \in \N}$ such that \begin{itemize}
\item $Q^{\ast}_i$ are open dyadic cubes;
\item $\cup_{i \in \N}~ \bar{Q}^{\ast}_i = X^C$;
\item $\dist (Q^{\ast}_i,X) \leq l(Q^{\ast}_i) \leq 4 \dist(Q^{\ast}_i,X)$,;
where $l(Q^{\ast}_i)$ denotes the sidelength of the cube.
\end{itemize}
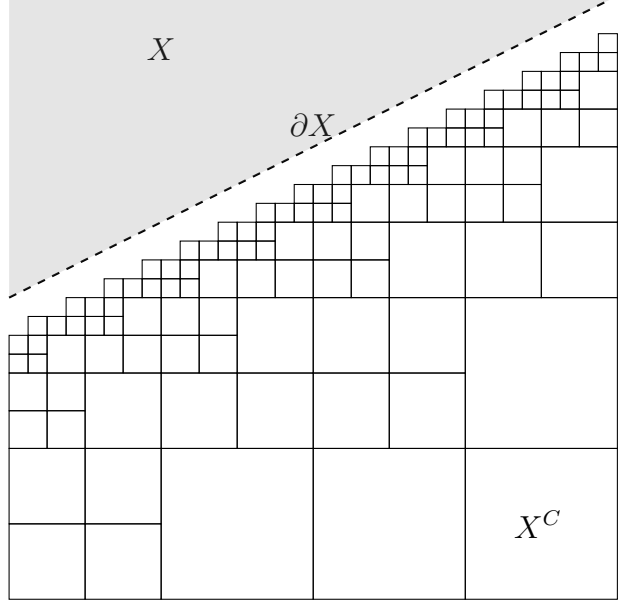
\begin{figure}
\centering
\begin{tikzpicture}[domain=0:8]
\draw[thick, dashed](0,4) -- (8,8]);
\coordinate[label=above:$\partial X$] (A) at (4,6);
\coordinate[label=above:$X$] (B) at (2,7);
\coordinate[label=right:$X^C$] (C) at (6.5,1);
\fill[color=gray, opacity=0.2] (0,4) -- (0,8) -- (8,8);
\foreach \x in {2,4,6} \draw (\x,0) rectangle (\x+2,2);
\draw (6,2) rectangle (8,4);
\foreach \x in {0,1} \foreach \y in {0,1} \draw (\x,\y) rectangle (\x+1,\y+1);
\foreach \x in {1,2} \draw (\x,2) rectangle (\x+1,3);
\foreach \x in {3,4,5} \foreach \y in {2,3} \draw (\x,\y) rectangle (\x+1,\y+1);
\foreach \x in {5,6,7} \draw (\x,4) rectangle (\x+1,5);
\draw (7,5) rectangle (8,6);
\foreach \x in {0,0.5} \foreach \y in {2,2.5} \draw (\x,\y) rectangle (\x+0.5,\y+0.5);
\draw (0.5,3) rectangle (1,3.5);
\draw (1,3) rectangle (1.5,3.5);
\foreach \x in {1.5,2,2.5} \foreach \y in {3,3.5} \draw (\x,\y) rectangle (\x+0.5,\y+0.5);
\draw (2.5,4) rectangle (3,4.5);
\draw (3,4) rectangle (3.5,4.5);
\foreach \x in {3.5,4,4.5} \foreach \y in {4,4.5} \draw (\x,\y) rectangle (\x+0.5,\y+0.5);
\draw (4.5,5) rectangle (5,5.5);
\draw (5,5) rectangle (5.5,5.5);
\foreach \x in {5.5,6,6.5} \foreach \y in {5,5.5} \draw (\x,\y) rectangle (\x+0.5,\y+0.5);
\foreach \x in {6.5,7,7.5} \draw (\x,6) rectangle (\x+ 0.5,6.5);
\draw (7.5,6.5) rectangle (8,7);
\foreach \x in {0,0.25} \foreach \y in {3,3.25} \draw (\x,\y) rectangle (\x+0.25,\y+0.25);
\foreach \x in {0.25,0.5,0.75} \draw (\x,3.5) rectangle (\x+0.25,3.75);
\draw (0.75,3.75) rectangle (1,4);
\foreach \x in {1,1.25} \foreach \y in {3.5,3.75} \draw (\x,\y) rectangle (\x+0.25,\y+0.25);
\foreach \x in {1.25,1.5,1.75} \draw (\x,4) rectangle (\x+0.25,4.25);
\draw (1.75,4.25) rectangle (2,4.5);
\foreach \x in {2,2.25} \foreach \y in {4,4.25} \draw (\x,\y) rectangle (\x+0.25,\y+0.25);
\foreach \x in {2.25,2.5,2.75} \draw (\x,4.5) rectangle (\x+0.25,4.75);
\draw (2.75,4.75) rectangle (3,5);
\foreach \x in {3,3.25} \foreach \y in {4.5,4.75} \draw (\x,\y) rectangle (\x+0.25,\y+0.25);
\foreach \x in {3.25,3.5,3.75} \draw (\x,5) rectangle (\x+0.25,5.25);
\draw (3.75,5.25) rectangle (4,5.5);
\foreach \x in {4,4.25} \foreach \y in {5,5.25} \draw (\x,\y) rectangle (\x+0.25,\y+0.25);
\foreach \x in {4.25,4.5,4.75} \draw (\x,5.5) rectangle (\x+0.25,5.75);
\draw (4.75,5.75) rectangle (5,6);
\foreach \x in {5,5.25} \foreach \y in {5.5,5.75} \draw (\x,\y) rectangle (\x+0.25,\y+0.25);
\foreach \x in {5.25,5.5,5.75} \draw (\x,6) rectangle (\x+0.25,6.25);
\draw (5.75,6.25) rectangle (6,6.5);
\foreach \x in {6,6.25} \foreach \y in {6,6.25} \draw (\x,\y) rectangle (\x+0.25,\y+0.25);
\foreach \x in {6.25,6.5,6.75} \draw (\x,6.5) rectangle (\x+0.25,6.75);
\draw (6.75,6.75) rectangle (7,7);
\foreach \x in {7,7.25} \foreach \y in {6.5,6.75} \draw (\x,\y) rectangle (\x+0.25,\y+0.25);
\foreach \x in {7.25,7.5,7.75} \draw (\x,7) rectangle (\x+0.25,7.25);
\draw (7.75,7.25) rectangle (8,7.5);
\end{tikzpicture} 
\caption{A collection of cubes $Q_j^{\ast}$ near the boundary (up to a certain size).} \label{picture2} \end{figure}

Choose $0 < \varepsilon < 1/4$ and define another collection of cubes by $Q_i= (1 + \varepsilon) Q_i^{\ast}$ (cube with the same center and sidelength $(1+\varepsilon) l(Q^{\ast}_i)$). Then 
\begin{itemize}
\item $\cup_{i \in \N}~ Q_i = X^C$;
\item For all $i \in \N$, the number of cubes $Q_j$ such that $Q_i \cap Q_j \neq \emptyset$ is bounded by a dimensional constant $C(N)$;
\item In particular, all $x \in \R^N$ are only contained in at most $C(N)$ cubes $Q_i$;
\item The distance to the boundary is again comparable to the sidelength, i.e. \begin{displaymath}
1/2 \dist (Q_i,X) \leq l(Q_i) \leq 8 \dist(Q_i,X). 
\end{displaymath}
\end{itemize}
Note that if $X$ is $\Z^N$-periodic, then also $Q_i$ can be chosen to be $\Z^N$ periodic (initially, we have a collection of dyadic cubes). 

Now consider $\phi \in C_c^{\infty}((-1-\varepsilon,1+\varepsilon)^N,[0,\infty))$ with $\phi =1$ on $(-1,1)^N$. We can rescale $\phi$ such that we obtain functions $\phi^{\ast}_j \in C_c^{\infty}(Q_j)$ with $\phi^{\ast}_j =1$ on $Q^{\ast}_j$. Define the partition of unity on $X^C$ by \begin{displaymath}
\phi_j = \frac{\phi^{\ast}_j}{\sum_{i \in \N} \phi_i^{\ast}}.
\end{displaymath}
Note that $0 \leq \phi_j \leq 1$ and that there exists a constant $C>0$ such that for all $j \in \N$ \begin{displaymath}
\V \nabla \phi_j \V \leq C/8~ l(Q_j)^{-1} \leq C \dist(Q_j,X)^{-1}.
\end{displaymath}
For each cube $Q_i$, we may find an $x \in X$ such that $\dist(Q_i,x) = \dist(Q_i,X)$. Denote this $x$ by $x_i$. For a multiindex $I=(i_1,...,i_{r+1}) \in \N^{r+1}$, define \begin{displaymath}
G(x_{i_1},...,x_{i_{r+1}})= G(I) := \fint_{\Sim(x_{i_1},...,x_{i_{r+1}})} \omega(y) \dy.
\end{displaymath}

We now define the differential form $\alpha \in L^1 (\R^N,\Lambda^r)$ by
\begin{equation} \label{alpha}
\alpha(y) := \sum_{I \in \N^{r+1}} \phi_{i_1} d\phi_{i_2}\wedge ... \wedge d\phi_{i_{r+1}} \wedge (G(I)(\nu^r(x_{i_1},...,x_{i_{r+1}}))).
\end{equation}
Note that in this setting $G(I)(\nu^r(...)) \in \R = \Lambda^0$.

We claim that the function $v \in L^1_{\loc} (\R^N,\Lambda^r)$ given by \begin{equation}
\label{vdef}
v(y) := \left\{ \begin{array}{ll} u(y) & y \in X, \\ (-1)^{r} \alpha(y) & y \in X^C \end{array} \right.
\end{equation}
is the function satisfying all the properties of Lemma \ref{extension}.

\begin{lemma} \label{aux1}
The differential form $\alpha$ defined in \eqref{alpha} satisfies $\alpha \in L^1(X^C, \Lambda^r)$ and the sum in  \eqref{alpha} converges pointwise and in $L^1$.
\end{lemma}

\begin{proof}
Pointwise convergence is clear, as for fixed $y \in X^C$ only finitely many summands are nonzero in a neighbourhood of $y$ ($\phi_i$ is only nonzero in $Q_i$ and any point is only covered by at most $C(N)$ cubes). For $L^1$ convergence fix some $i_1 \in \N$. Note that there are at most $C(N)^r$ summands in $i_2,...,i_{r+1}$, which are nonzero, as $Q_{i_1}$ only intersects with $C(N)$ other cubes. Furthermore, note that for all $i_l$ with $Q_{i_l} \cap Q_{i_1} \neq \emptyset$ \begin{displaymath}
\norm d \phi_{i_l} (y) \norm_{\Lambda^1} \leq C \dist(y,X) \leq C l(Q_{i_1})^{-1}.
\end{displaymath}
Moreover, we can bound $\nu^r$ by \begin{displaymath}
\norm \nu^r(x_{i_1},...,x_{i_{r+1}}) \norm_{\Lambda_r} \leq \max_{a,b \in \{i_1,..,i_{r+1}\}} \V x_a -x_b \V^r \leq C l(Q_{i_1})^r.
\end{displaymath}
Hence, we can bound the $L^{\infty}$-norm of a nonzero summand of \eqref{alpha} by $C \norm u \norm_{L^{\infty}}$, as $\V G(I) \V \leq \norm u \norm_{L^{\infty}}$. As the support of the summand is contained in $Q_i$, we have that its $L^1$ norm is bounded by \begin{displaymath}
C \norm u \norm_{L^{\infty}} \V Q_{i_1} \V.
\end{displaymath}
Remember that any point in $X^C$ is covered by only $C(N)$ cubes, such that the sum of $\V Q_i \V$ is bounded by $C(N) \V X^C \V$. Hence, the sum in \eqref{alpha} converges absolutely in $L^1$ and its $L^1$ norm is bounded by $C(N)^{r+1} C \norm u \norm_{L^{\infty}} \V X^C \V$.
\end{proof}

\begin{lemma} \label{aux2}
The function $v$ is strongly differentiable almost everywhere and satisfies $dv(y) =0$ for all $y \in \R^N \back \partial X$.
\end{lemma}

\begin{proof}
Note that $u \in C_c^{\infty}(\R^N,\Lambda^r)$ and hence $v$ is strongly differentiable in $X \back \partial X$. Furthermore, the sum in \eqref{alpha} is a finite sum in a neighbourhood of $y$ for all $y \in X^C$. As the summands are also $C^{\infty}$, the sum is $C^{\infty}$ in the interior of $X^C$.

By assumption $du=0$, hence it remains to prove that $d \alpha (y) = 0$ for all $y \in X^C$. Note that in a neighbourhood of $y \in X^C$ again only finitely many summands are nonzero. Using that $d^2=0$ and the Leibniz rule, we get
\begin{equation} \label{dalpha}
d \alpha(y) = \sum_{I \in \N^{r+1}}   d\phi_{i_1}(y) \wedge ... \wedge  d \phi_{i_{r+1}}(y) (G(I)(\nu^r(x_{i_1},...,x_{i_{r+1}}))) .
\end{equation}

Observe  that this term does not converge in $L^1$ and hence this identity is only valid pointwise.

Pick some $j \in \N$ such that $y \in Q_j$. As all $\phi_i$ sum up to $1$ in $X^C$, we have \begin{displaymath}
d\phi_j(y) = - \sum_{I \in \N \back \{j\}} d\phi_i(y).
\end{displaymath}
Replace $d \phi_j$ in the sum in \eqref{dalpha} by $- \sum_{I \in \N \back \{j\}} d\phi_i(y)$. Recall that $\nu^r(x_1,...,x_{r+1})=0$ if  $x_l=x_{l'}$ for some $l \neq l'$. Hence, \begin{align*}
d \alpha(y) &= \sum_{I \in \N^{r+1}} d \phi_{i_1}(y) \wedge ... \wedge d\phi_{i_{r+1}}(y) \wedge (G(I)(\nu^r(x_{i_1},...,x_{i_{r+1}})))\\
~ & =\sum_{I \in (\N \back \{j\})^{r+1}}  d \phi_{i_1}(y) \wedge ... \wedge  d \phi_{i_{r+1}}(y) \wedge(G(I)(\nu^r(x_{i_1},...,x_{i_{r+1}})))\\
~&~~ - \sum_{l=1}^{r+1} \sum_{I \in (\N \back \{j\})^{r+1}}  d \phi_{i_1(y)} \wedge ... \wedge  d \phi_{i_{r+1}}(y) \\ & \hspace{1.5cm}\wedge (G(x_{i_1},...x_{i_{l-1}},x_j,x_{i_{l+1}},...)( \nu^r(x_{i_1},...x_{i_{l-1}},x_j,x_{i_{l+1}},...))).
\end{align*}

We apply Stoke's theorem \eqref{Stokes2} to the $r$-form $u$ and the simplex with vertices $x_j, x_{i_1},...,x_{i_{r+1}}$, use that $d u=0$ and conclude that this term is $0$, i.e. \begin{align*}
G(I)&(T^r(x_{i_1},...,x_{i_{r+1}})) - \sum_{l=1}^{r+1} G(x_{i_1},...,x_{i_{l-1}},x_j,x_{i_{l+1}},...)( \nu^r(x_{i_1},...,x_j,x_{i_{l+1}},...)) \\ ~& = - \frac{r-1}{r} \fint_{\Sim(x_j,x_{i_1},...,x_{i_{r+1}})} du(y)( \nu^{r+1}(x_j,x_{i_1},...,x_{i_{r+1}})) \dH^{r}(y) = 0.
\end{align*}

Hence, the pointwise derivative equals $0$ almost everywhere.
\end{proof}
It is important to note that the sum \eqref{alpha} in the definition of $\alpha$ converges in $L^1$, but in general does not converge in $W^{1,1}$, and thus we have no information on the behaviour at the boundary of $X^C$. However, it suffices to show that the distribution $dv$ for $v$ given by \eqref{vdef} is actually an $L^1$ function. If $d v \in L^1$, we can conclude with Lemma \ref{aux2} that $dv= 0$ in the sense of distributions.

\begin{lemma} \label{aux3} 
The distributional exterior derivative of $v$ defined in \eqref{vdef} satisfies $dv \in L^1(\R^N,\Lambda^{r+1})$, i.e. there exists an $L^1$ function $h \in L^1(\R^N,\Lambda^{r+1})$ such that for all $\psi \in C_c^{\infty}(\R^N,\Lambda^{N-r-1})$ \begin{displaymath}
(-1)^{r} \int_{X^C} \alpha \wedge d\psi  + \int_{X} u \wedge d\psi  = \int_{\R^N} h \wedge \psi.
\end{displaymath}
\end{lemma}
\begin{proof}
Consider \begin{displaymath}
\int_{X^C} \alpha(y) \wedge d\psi(y) \dy.
\end{displaymath}
 In view of the definition of $\alpha$, this expression is given by:
 \begin{align*}
\int_{\R^N} & \sum_{I \in \N^{r+1}}  \phi_{i_1}d\phi_{i_2}\wedge ... \wedge \phi_{i_{r+1}}  (G(I)(\nu^r(x_{i_1},...,x_{i_{r+1}})))d\psi~\dy = (\ast).
\end{align*}

We use the splitting $G(I) = (G(I) - u(\cdot)) +u(\cdot)$ and write ($\ast$) as
\begin{equation} \label{splitting}
\begin{aligned}
(\ast) &=  \int_{\R^N} \sum_{I \in \N^{r+1}} \phi_{i_1} d\phi_{i_2}\wedge ...\wedge d\phi_{i_{r+1}} \wedge ((G(I) - u(\cdot))( \nu^r(x_{i_1},...,x_{i_{r+1}})) \wedge d\psi\\ 
~& ~+ \int_{\R^N} \sum_{I \in \N^{r+1}} \phi_{i_1} d \phi_{i_2}\wedge...\wedge d \phi_{i_{r+1}}\wedge ( u(\cdot)(\nu^r(x_{i_1},...,x_{i_{r+1}})))\wedge d\psi~ \\
~& = \text {(I)} + \text {(II)}.
\end{aligned}
\end{equation}
Note that (I) defines a distribution given by an $L^1$ function. Indeed, the sum \begin{displaymath}
\phi_{i_1} d\phi_{i_2}\wedge ...\wedge d\phi_{i_{r+1}} \wedge ((G(I) - u(y))( \nu^r(x_{i_1},...,x_{i_{r+1}}))
\end{displaymath} converges in $W^{1,1}(\R^N,\Lambda^{r+1})$. To see this, one can repeat the proof of Lemma \ref{aux1} and use that there are additional factors in the estimate of the norms. For this, note that if $z \in Q_{i_1}$ \begin{displaymath}
\norm G(I) - u(z) \norm_{\Lambda^r} \leq Cl(Q_i)\norm \nabla u \norm_{L^{\infty}} 
\end{displaymath} 
and \begin{displaymath}
\norm \nabla (G(I) -u(\cdot))(z) \norm_{\Lambda^r} \leq C  \norm\nabla u \norm_{L^{\infty}}.
\end{displaymath}
One gets improved regularity and may integrate by parts to eliminate the derivative of $\psi$.

Term (II) is not so easy to handle. We prove the following claims:

\noindent \textbf{Claim 1:} Let $1 \leq s \leq r$ and $I' =(i_s,...,i_{r+1}) \in \N^{r-s+2}$. There exists $h_s \in L^1(T_N,\Lambda^{r+1})$ such that \begin{equation} \label{indstep} \begin{aligned}
\int_{X^C} & \sum_{I' \in \N^{r-s+2}} \phi_{i_s} d\phi_{i_{s+1}} \wedge...\wedge d \phi_{i_{r+1}} \wedge (u(\cdot)(\nu^{r-s+1}(x_{i_s},...,x_{i_{r+1})}))) \wedge d \psi \\
&= \int_{X^C} h_s \wedge \psi \\
& \quad - \int_{X^C} \sum_{I' \in \N^{r-s+1}} \phi_{i_{s+1}} d \phi_{i_{s+2}} \wedge ... \wedge d\phi_{i_{r+1}} \wedge (u(\cdot)(\nu^{r-s}(x_{i_{s+1}},...,x_{i_{r+1}})) \wedge d \psi.
\end{aligned}
\end{equation}
Here we use the notation that $\nu^0(x_{i_{r+1}}) = 1 \in \Lambda_0 =\R$.

\noindent \textbf{Claim  2:} There is $\tilde{h} \in L^1(T_N,\Lambda^{r+1})$ such that  \begin{equation} \label{ind} \begin{aligned}
\int_{\R^N} & \sum_{I' \in \N^{r+1}} \phi_{i_1} d\phi_{i_{2}} \wedge...\wedge d \phi_{i_{r+1}} \wedge (u(\cdot)(\nu^{r}(x_{i_1},...,x_{i_{r+1})}))) \wedge d \psi \\
&= \int_{X^C} \tilde{h} \wedge \psi + (-1)^r \int_{X^C} u \wedge d\psi.\\
\end{aligned}
\end{equation}

Note that Claim 2 follows from Claim 1 by an inductive argument. The domain of integration in \eqref{ind} can be replaced by $X^C$ as well, as all $\phi_{i_j}$ are supported in $X^C$.

First, let us conclude the proof under the assumption that Claim 1 holds true. Using \eqref{splitting} and Claim 2 we see that there is an $h \in L^1(\R^N,\R^d)$ such that \begin{displaymath}
\int_{X^C} \alpha \wedge d\psi = \int_{\R^N} h \wedge \psi + (-1)^{r}\int_{X^C} u \wedge d\psi.
\end{displaymath}
Recall that $du =0$ in the sense of distributions and therefore 
\begin{displaymath}
-\int_{X^C}  u\wedge d \psi  =  \int_{X} u\wedge d \psi.
\end{displaymath}
We conclude that there exists an $L^1$ function $h \in L^1(\R^N,\Lambda^{r+1})$ such that \begin{displaymath}
\int_{X^C} \alpha \wedge d\psi +   (-1)^{r} \int_{X} u \wedge  d\psi  = \int_{\R^N} h \wedge \psi.
\end{displaymath}
Thus, $dv$ is an $L^1$ function.

It remains to prove Claim 1. Note that 
 \begin{equation} \label{identity1}
\nu^{r-s+1}(x_{i_s},...,x_{i_{r+1}}) = \sum_{j=s}^{r+1} \nu^{r-s+1}(x_{i_s},...,x_{i_{j-1}},y,x_{i_{j+1}},...,x_{i_{r+1}}).
\end{equation}
This can be verified using that the wedge product is alternating and explicitly writing the right-hand side of \eqref{identity1}.

Using this identity, we may split the right-hand side of \eqref{indstep} (denoted by (III)), i.e.
\begin{align*}
\text{(III)} &= \sum_{j=s+1}^{r+1} \int_{\R^N} \sum_{I \in \N^{r-s+2}}  \phi_{i_s} d\phi_{i_{s+1}}\wedge ... \wedge d\phi_{i_{r+1}}\\ & \hspace{2cm} \wedge (u(\cdot)( \nu^{r-s+1}(x_{i_s},...,x_{i_{j-1}},y,x_{i_{j+1}},...,x_{i_{r+1}}))) \wedge d\psi \\~& +\int_{\R^N} \sum_{I \in \N^{r-s+2}}   \phi_{i_s} d\phi_{i_{s+1}}\wedge... \wedge d\phi_{i_{r+1}}\wedge( u(\cdot) (\nu^{r-s+1}(y,x_{i_{s+1}},...,x_{i_{r+1}})))\wedge d\psi \\ ~&= \text {(IIIa)} + \text{(IIIb)}.
\end{align*}
Arguing as in Lemma \ref{aux1}, we see that the sum \begin{displaymath}
\sum_{I \in \N^{r-s+2}}  \phi_{i_s} d\phi_{i_{s+1}}\wedge ... \wedge d\phi_{i_{r+1}} \wedge (u(\cdot)( \nu^{r-s+1}(x_{i_s},...,x_{i_{j-1}},y,x_{i_{j+1}},...,x_{i_{r+1}}))) 
\end{displaymath}
is in fact convergent in $L^1$. Moreover, the index $i_j$ only appears once in this sum. Recall that for $y \in X^C $ \begin{displaymath}
\sum_{i_j \in \N} d\phi_{i_j}(y) =0.
\end{displaymath}
Thus, \begin{displaymath}
\text{(IIIa)} = 0.
\end{displaymath}
For (IIIb) note that $\sum_{i_1 \in \N} \phi_{i_s} = 1_{X^C}$ and, by the same argument as for (IIIa), we can write \begin{align*}
\text{(IIIb)} &= \int_{X^C} \sum_{I \in \N^{r-s+1}}  d\phi_{i_{s+1}}\wedge...\wedge d\phi_{i_{r+1}}\wedge( u(\cdot)(\nu^{r-s+1}(y,x_{i_{s+1}},...,x_{i_{r+1}}))) \wedge d\psi.
\end{align*}
We can now integrate by parts to eliminate the exterior derivative in front of $\phi_{i_{s+1}}$.  Applying Lemma \ref{prodrule}, using $d^2=0$, the Leibniz rule and the fact that $\phi_{i_j} \in C_c^{\infty}(\R^N,\R)$  \begin{align*}
&(-1)^{r-s+1} \text{(IIIb)}\\ &=  \int_{X^C}  \sum_{I \in \N^{r-s+1}}    \phi_{i_{s+1}}d\phi_{i_{s+2}}\wedge... \wedge d\phi_{i_{r+1}} \wedge d( u(\cdot)( \nu^{r-s+1}(y,x_{i_{s+1}},...,x_{i_{r+1}})))\wedge d\psi  \\ 
~ & = \int_{X^C}  \sum_{I \in \N^{r-s+1}}    \phi_{i_{s+1}}d\phi_{i_{s+2}}\wedge... \wedge d\phi_{i_{r+1}} \\ & \hspace{4cm} \wedge D^{r-s+1,r}(\nabla u(\cdot),( \nu^{r-s+1}(y,x_{i_{s+1}},...,x_{i_{r+1}})))\wedge d\psi \\
~& ~+~(-1)^{(r-s)} \int_{X^C } \sum_{I \in \N^{r-s+1}}    \phi_{i_{s+1}}d\phi_{i_{s+2}}\wedge... \wedge d\phi_{i_{r+1}} \\ &\hspace{4cm} \wedge u(\cdot)( \nu^{r-s}(x_{i_{s+1}},...,x_{i_{r+1}})))\wedge d\psi  \\
 ~&= \text{(IIIc)} +\text{(IIId)}.
\end{align*}
Arguing similarly to Lemma \ref{aux1} and as for term (I), we can show that 
\begin{displaymath}
 \sum_{I \in \N^{r}}    \phi_{i_{s+1}}d\phi_{i_{s+2}}\wedge... \wedge \phi_{i_{r+1}} \wedge D^{r-s+1,r}(\nabla u(\cdot),( \nu^{r-s+1}(y,x_{i_{s+1}},...,x_{i_{r+1}}))) \in W^{1,1}(\R^N,\Lambda^r),
\end{displaymath}
and that this sum is convergent in $W^{1,1}$.
Hence, we have shown that there exists $h_s \in L^1(\R^N,\Lambda^{r+1})$ such that \begin{equation} \label{induction} \begin{aligned} 
\text{(III)}=& \int_{\R^N} \sum_{I \in \N^{r-s+2}} \phi_{i_s} d \phi_{i_{s+1}}\wedge...\wedge d \phi_{i_{r+1}}\wedge ( u(\cdot)(\nu^{r-s+1}(x_{i_1},...,x_{i_{r+1}})))\wedge d\psi~    \\
&= \int_{\R^N} h_s \wedge \psi  
\\ &\hspace{0.5cm} -\int_{\R^N} \sum_{I \in \N^{r-s+1}} \phi_{i_{s+1}} d\phi_{i_{s+2}} \wedge ... \wedge d\phi_{i_{r+1}} (u(\cdot)(\nu^{r-s}(x_{i_{s+1}},...,x_{i_{r+1}})))\wedge d \psi
\end{aligned}
\end{equation}
Hence, Claim 1 holds, completing the proof of Lemma \ref{aux3}
\end{proof}
This proves Lemma \ref{extension}. The property that \begin{displaymath}
 dv = 0 \quad \text{in the sense of distributions}
\end{displaymath} 
follows from Lemma \ref{aux2} and Lemma \ref{aux3}. By definition, $v=u$ on $X$. Finally, we can bound the $L^{\infty}$-norm of $v$ by $CL$, as in the definition of $\alpha$ \begin{displaymath}
 \sum_{I \in \N^{r+1}} \phi_{i_1} d\phi_{i_2}\wedge ... \wedge d\phi_{i_{r+1}} \wedge (G(I)(\nu^r(x_{i_1},...,x_{i_{r+1}})))
\end{displaymath}
every summand can be bounded by $CL$ due to (\ref{Lipschitzprop}) and the estimate $\V d \phi_j \V \leq C \dist(Q_j,X)^{-1}$. Again, we get the $L^{\infty}$ bound, as only finitely many summands on are nonzero for every $y \in X^C$.

With slight modifications one is able to prove the following variants.

\begin{coro} \label{unbounded}
Let $u \in C^{\infty}(\R^N,\Lambda^r)$ with $du=0$, let $L>0$, and let $X \subset \R^N$ be a nonempty closed set such that $\norm u \norm_{L^{\infty}(X)} \leq L$ and for all $x_1,...,x_{r+1} \in X$ we have  \begin{displaymath}
\left \V \fint_{\Sim(x_1,...,x_{r+1})}  u(y)(\nu^r(x_1,...,x_{r+1})) \dy \right \V \leq L \max \V x_i -x_j \V ^r.
\end{displaymath}
Suppose further that $\V \partial X \V =0$.

There exists a constant $C= C(N,r)$ such that for all  $u \in C^{\infty}(\R^N,\Lambda^r)$ and $X$ meeting these requirements there exists $v \in L^1_{\loc}(\R^N,\Lambda^r)$ with \begin{enumerate} [i)]
\item $dv =0$ in the sense of distributions;
\item $v(y)=u(y)$ for all $y \in X$;
\item $\norm v \norm_{L^{\infty}} \leq CL$.
\end{enumerate}
\end{coro}

This statement is proven in the same way as Lemma \ref{extension}, but all the statements are only true locally (e.g. the $L^1$ bounds on $\alpha$ are replaced by bounds in $L^1_{\loc}(X^C,\Lambda^r)$).

If we choose $u$ and $X$ to be $\Z^N$ periodic we get a suitable statement for the torus.

\begin{coro} \label{torus}
Let $u \in C^{\infty}(T_N,\Lambda^r)$ with $du =0$, let $L>0$, and let $X \subset \R^N$ be a nonempty, closed, $\Z^N$-periodic set (which can be viewed as a subset of $T_N$) such that $\norm u \norm_{L^{\infty}(X)} \leq L$ and for all $x_1,...,x_{r+1} \in X$ we have  \begin{displaymath}
\left \V \fint_{\Sim(x_1,...,x_{r+1})}  \tilde{u}(y)(\nu^r(x_1,...,x_{r+1}))  \dy\right \V \leq L \max \V x_i -x_j \V ^r,
\end{displaymath}
where $\tilde{u}\in C^{\infty}(\R^N,\Lambda^r)$ is the $\Z^N$-periodic representative of $u$.
Suppose further that $\V \partial X \V =0$.

There exists a constant $C= C(N,r)$ such that for all  $u \in C^{\infty}(T_N,\Lambda^r)$ and $X$ meeting these requirements there exists $v \in L^1(T_N,\Lambda^r)$ with \begin{enumerate} [i)]
\item $dv =0$ in the sense of distributions;
\item $v(y)=u(y)$ for all $y \in X \subset T_N$;
\item $\norm v \norm_{L^{\infty}} \leq CL$.
\end{enumerate}
\end{coro}
As mentioned before, we can choose the cubes $Q_j$ to be rescaled dyadic cubes. As the set $X$ is periodic, the set of cubes (and hence also the partition of unity) and their projection points may also be chosen to be $\Z^N$-periodic. By definition then also the the extension will be $\Z^N$-periodic.


\section{$L^{\infty}$-truncation} \label{Secmain}
Now we prove the main result of this paper on the $L^{\infty}$-truncation of closed forms.
\begin{thm}[$L^{\infty}$-truncation of differential forms] \label{main}
There exist constants $C_1,C_2 >0$ such that for all  $u \in L^1(T_N,\Lambda^r)$ with $du=0$ and all $L>0$ there exists $v \in L^{\infty}(T_N,\Lambda^r)$ with $dv=0$ and \begin{enumerate} [i)]
\item $\norm v \norm_{L^{\infty}(T_N,\Lambda^r)} \leq C_1 L$;
\item $\V \{y \in T_N \colon v(y) \neq u(y) \V\} \leq \frac{C_2}{L} \int_{\{y \in T_N\colon  \V u(y) \V >L\}} \V u(y) \V \dy$;
\item $\norm v - u \norm_{L^1(T_N,\Lambda^r)} \leq C_2 \int_{\{y \in T_N \colon  \V u(y) \V >L\}} \V u(y) \V \dy$.
\end{enumerate}
\end{thm}
Given the Whitney-type extension obtained in Lemma \ref{torus} and  Lemma \ref{extension} combined with Lemma \ref{maximalf}, the proof now roughly follows \textsc{Zhang}'s proof for Lipschitz truncation in \cite{Zhang}. First, we prove the statement in the case that $v$ is smooth directly using our extension theorem for the set $X= \{M u \leq L\}$. After calculations similar to \cite{Zhang} we are able to show that this extension satisfies the properties of Theorem \ref{main}. Afterwards, we prove the statement for $u \in L^1(T_N,\Lambda^r)$ by a standard density argument.

\begin{proof}
First, suppose that $u \in C^{\infty}(T_N,\Lambda^r)$. For $\lambda>0$ define the set \begin{displaymath} X_{\lambda} = \{y \in T_N \colon Mu(y) \leq \lambda\}.
\end{displaymath}
 Choose $2L \leq \lambda \leq 3L$ such that $\V \partial X_{\lambda} \V =0$. Then, by Lemma \ref{maximalf} and the extension Lemma \ref{torus}, there exists a $v \in L^1(T_N,\Lambda^r)$ with \begin{enumerate}
\item $\{ y \in T_N \colon v(y) \neq u(y) \} \subset X_{\lambda}^C$.
\item $\norm v \norm_{L^{\infty}} \leq CL$.
\item $dv =0$ in the sense of distributions.
\end{enumerate}
We need to show that \begin{equation} \label{TS1}\norm v - u \norm_{L^1(T_N,\Lambda^r)} \leq C_2 \int_{\{y\colon  \V u(y) \V >L\}} \V u(y) \V \dy \end{equation} 
and that \begin{equation} \label{TS2}
\V \{y \in T_N \colon v(y) \neq u(y) \}\V \leq \frac{C_2}{L} \int_{\{y\colon  \V u(y) \V >L\}} \V u(y) \V \dy. \end{equation}

Indeed, \eqref{TS1} follows from \eqref{TS2}, as 
\begin{align*}
\int_{T_N} \V v(y) -u(y) \V \dy &= \int_{X_{\lambda}^C} \V v(y) -u(y)\V \dy \\
~& \leq \int_{\{Mu \geq \lambda\}} \V u(y) \V + \int_{\{Mu \geq \lambda\}} \V v(y) \V \dy \\
~& \leq  \int_{\{\V u \V \geq \lambda\}} \V u(y) \V \dy + 2CL \V \{ Mu \geq \lambda\} \V.
\end{align*}
Thus, it suffices to prove \eqref{TS2}.

To this end, define the function $h \colon \Lambda^r \to \R$ by \begin{displaymath}
h(z) = \left \{ \begin{array}{ll} 0 & \text{if } \V z \V <L, \\ \V z \V -L & \text{if } \V z \V \geq L. \end{array} \right.
\end{displaymath}
Let $y \in \{Mu > \mu\}$ for $\mu \in \R$. Then there exists an $R>0$ such that \begin{displaymath}
\fint_{B_R(y)} \V u(z)\V  \dz > \mu.
\end{displaymath}
Thus, \begin{align*}
M (h (u))(y) & \geq \fint_{B_R(y)} \V h(u) (z) \V \dz \\
~& = \frac{1}{\V B_R(y) \V} \int_{B_R(y) \cap \{u \geq L\}} \V u(z) \V - L \dz \\
~& \geq \fint_{B_R(y)} \V u(z) \V \dz - \frac{1}{\V B_R(y) \V} \int_{B_R(y) \cap \{u \leq L\}} \V u(z) \V \dz \\ & \hspace{1cm}- \frac{1}{\V B_R(y) \V} \int_{B_R(y)  \cap \{ \V u \V \geq L\}} L \dz \\
~& \geq \mu -L.
\end{align*}
Thus, $\{y \in T_N \colon Mu > \mu\} \subset \{y \in T_N \colon M h(u) (y) > \mu - L\}$.

Using the weak-$L^1$ estimate for the maximal function (Proposition \ref{maximalfct}), we get \begin{equation} \label{estimatex} \begin{aligned}
\V \{ y \in T_N \colon Mu(y) \geq \lambda\} \V & \leq \V \{ y \in T_N \colon M h(u) \geq \lambda-L\} \V\\
~&\leq \frac{1}{\lambda -L} C \int_{T_N} \V h(u)(z) \V \dz\\
~& \leq \frac{C}{L} \int_{T_N \cap \{\V u \V \geq L\}} \V u(z) \V \dz.
\end{aligned}
\end{equation}
This is what we wanted to show. Note that the proof only uses $u \in C^{\infty}(T_N,\Lambda^r)$ to define $v$ and nowhere else, hence estimate \eqref{estimatex} is valid for all $u \in L^1(T_N,\Lambda^r)$.

For general $u \in L^1(T_N,\Lambda^r)$, one may consider a sequence $u_n \in C^{\infty}(T_N,\Lambda^r)$ with $d u_n=0$ and $u_n \to u$ in $L^1$ and pointwise almost everywhere. This sequence can be easily constructed by convolving with standard mollifiers.

Observe that for $\lambda >0$ \begin{align} \label{estimate1}
\int_{\{\V u_n \V \geq 2\lambda\}} \V u_n \V \dy & \leq \int_{\{ \V u_n -u \V \geq \V u \V\} \cap \{ \V u_n \V \geq 2 \lambda\}} \V u_n \V \dy + \int_{\{\V u_n - u \V \leq \V u \V \} \cap \{ \V u_n \V \geq 2 \lambda\}} \V u_n \V \dy \\ \nonumber
~& \leq 2 \int_{\{\V u \geq \lambda\V\}} \V u \V \dy +2 \norm u_n - u \norm_{L^1}.
\end{align}

Furthermore, we use the subadditivity of the maximal function and see that for all $y \in T_N$ \begin{displaymath}
M u_n (y) \leq M u (y) + M(u - u_n)(y).
\end{displaymath}
Thus, \begin{displaymath}
\{y \in T_N \colon M u_n(y) \geq 2\lambda\} \subset \{y \in T_N \colon M u(y) \geq \lambda\} \cup \{ y \in T_N \colon M (u - u_n)(y) \geq \lambda\}.
\end{displaymath} 
Using the weak-$L^1$estimate for the maximal function (Proposition \ref{maximalfct}) we see that \begin{equation} \label{setestimate}
\left \V \{ y \in T_N \colon M u(y) \leq \lambda\} \back \{y \in T_N \colon M u_n(y)  \geq 2 \lambda\} \right \V \longrightarrow 0 \quad \text{as } n  \to  \infty.
\end{equation}
Choose some $\lambda \in (4L,6L)$ such that for all $n \in \N$ $\V \partial \{y \in T_N \colon Mu_n(y) \geq 2\lambda\} \V =0$. Then extend like in the first part of the proof to get a sequence $v_n$ with $dv_n=0$ and
 \begin{enumerate} [a)]
\item $\norm v_n \norm_{L^{\infty}(T_N,\Lambda^r)} \leq 2C_1 \lambda$;
\item $\V \{y \in T_N \colon v_n(y) \neq u_n(y) \V\} \leq \frac{C_2}{2\lambda} \int_{y\colon  \V u_n(y) \V >2 \lambda} \V u_n(y) \V \dy$;
\item $\norm v_n - u_n \norm_{L^1(T_N,\Lambda^r)} \leq C_2 \int_{\{y\colon  \V u_n(y) \V > 2 \lambda\}} \V u_n(y) \V \dy$.
\end{enumerate}
Letting $n \to \infty$, by a) this sequence converges, up to extraction of a subsequence, weakly$*$ to some $v \in L^{\infty}(T_N,\Lambda^r)$. The weak$*$-convergence implies $dv =0$. Moreover, by construction, the set $\{y \in T_N \colon v_n \neq u_n\}$ is contained in the set $\{y \in T_N \colon Mu_n(y) \geq 2 \lambda\}$. As $u_n \to u$ pointwise  a.e. and in $L^1$,
 we get using \eqref{setestimate} that $v = u$ on the set $\{y \in T_N \colon M u(y) \leq \lambda\}$. (If $v_n$ converges to $u$ in measure on a set $A$ and $v_n$ weakly to some $v$, then $v=u$ on $A$.)

Hence, $v$ defined as the weak$*$ limit of $v_n$ satisifies \begin{enumerate} [i)]
\item $\norm v \norm_{L^{\infty}(T_N,\Lambda^r)} \leq C_1 \lambda \leq 6 C_1 L $;
\item using \eqref{estimatex} and $v=u$ on $\{y \in T_N \colon M u(y) \leq \lambda\}$ \begin{displaymath}
 \V \{ y \in T_N \colon u(y) \neq v(y) \} \V \leq \frac{C_2}{L} \int_{\{y \in T_N\colon  \V u(y) \V >L\} } \V u(y) \V \dy; 
\end{displaymath}
\item using triangle inequality and $v_n - u_n \to 0$ in $L^1$, one obtains \begin{displaymath}
\norm v - u \norm_{L^1(T_N,\Lambda^r)} \leq C_2 \int_{\{y \in T_N \colon  \V u(y) \V >L\}} \V u(y) \V \dy.
\end{displaymath}
\end{enumerate}
Hence, $v$ meets the requirements of Theorem \ref{main}.
\end{proof}
 
\begin{coro}[$L^\infty$-truncation for sequences] \label{main2}
Suppose that we have a sequence $u_n \in L^1(\R^N,\Lambda^r)$ with $d u_n=0$, and that there exists $L>0$ such that \begin{displaymath}
\int_{\{y \in T_N \colon \V u_n(y) \V >L \}} \V u_n(y) \V \dy \longrightarrow 0 \quad \text{as } n \to \infty.
\end{displaymath}
There exists a $C_1=C_1(N,r)$ and a sequence $v_n \in L^1(T_N,\Lambda^r)$ with $dv_n=0$ and \begin{enumerate} [a)]
\item $\norm v_n \norm_{L^{\infty}(T_N,\Lambda^r)} \leq C_1 L$;
\item $\norm v_n - u_n \norm_{L^1(T_N,\Lambda^r)} \to 0$ as $n \to \infty$;
\item $\V \{ y \in T_N \colon v_n(y) \neq u_n(y) \} \V \to 0$.
\end{enumerate}
\end{coro}
This directly follows by applying Theorem \ref{main}.

The proof of Theorem \ref{main} also works if $L^1$ is replaced by $L^p$ for $1<p< \infty$. Furthermore, we do not need to restrict us to periodic functions on $\R^N$, the statement is also valid for non-periodic functions.

\begin{prop} \label{main3}
Let $1 \leq p < \infty$. There exist constants $C_1,C_2 >0$, such that, for all  $u \in L^p(\R^N,\Lambda^r)$ with $du=0$ and all $L>0$, there exists $v \in L^p(\R^N,\Lambda^r)$ with $dv=0$ and \begin{enumerate} [i)]
\item $\norm v \norm_{L^{\infty}(\R^N,\Lambda^r)} \leq C_1 L$;
\item $\V \{y \in \R^N \colon v(y) \neq u(y) \V\} \leq \frac{C_2}{L^p} \int_{\{y \in \R^N \colon  \V u(y) \V >L\}} \V u(y) \V^p \dy$;
\item $\norm v - u \norm^p_{L^p(\R^N,\Lambda^r)} \leq C_2 \int_{\{y \in \R^N \colon  \V u(y) \V >L\}} \V u(y) \V^p \dy$.
\end{enumerate}
\end{prop}
As described, the proof is pretty much the same as for Theorem \ref{main}.
We may also want to truncate closed forms supported on an open bounded subset $\Omega \subset \R^N$ (c.f. \cite{Sebastian,Solenoidal}). This is possible, but we may lose the property, that they are supported in this subset. Let us, for simplicity, consider balls $\Omega=B_{\rho}(0)$ and, after rescaling, $\rho=1$.

\begin{prop}\label{main4}
Let $1 \leq p < \infty$. There exist constants $C_1,C_2 >0$ such that, for all  $u \in L^p( \R^N,\Lambda^r)$ with $du=0$ and $\spt(u) \subset B_1(0)$ and all $L>0$, there exists $v \in L^p(\R^N,\Lambda^r)$ with $dv=0$ and \begin{enumerate} [i)]
\item $\norm v \norm_{L^{\infty}(\R^N,\Lambda^r)} \leq C_1 L$;
\item $\V \{y \in \R^N \colon v(y) \neq u(y) \V\} \leq \frac{C_2}{L^p} \int_{\{y \in \R^N \colon  \V u(y) \V >L\}} \V u(y) \V^p \dy$;
\item $\norm v - u \norm^p_{L^p(\R^N,\Lambda^r)} \leq C_2 \int_{\{y \in \R^N \colon  \V u(y) \V >L\}} \V u(y) \V^p \dy$;
\item $\spt(v) \subset B_R(0)$, where $R$ only depends on the $L^p$-norm of $u$ and on $L$.
\end{enumerate}
\end{prop}
Again, this proof is very similar to the proof of Theorem \ref{main}. Property iv) comes from the fact that if a function $u$ is supported in $B_1(0)$, then its maximal function $Mu(y)$ decays fast as $y \to \infty$. Let us mention that this result also holds for vector-valued differential forms, i.e. $u \in L^p(\R^N,\Lambda^r \times \R^m)$, where the exterior derivative is taken componentwise.
\begin{prop}[Vector-valued forms on the torus] \label{main5}
There exist constants $C_1,C_2 >0$ such that, for all  $u \in L^1(T_N,\Lambda^r\times\R^m)$ with $du=0$ and all $L>0$, there exists $v \in L^1(T_N,\Lambda^r \times \R^m)$ with $dv=0$ and \begin{enumerate} [i)]
\item $\norm v \norm_{L^{\infty}(T_N,\Lambda^r \times \R^m)} \leq C_1 L$;
\item $\V \{y \in T_N \colon v(y) \neq u(y) \V\} \leq \frac{C_2}{L} \int_{\{y \in T_N\colon  \V u(y) \V >L\}} \V u(y) \V \dy$;
\item $\norm v - u \norm_{L^1(T_N,\Lambda^r \times \R^m)} \leq C_2 \int_{\{y \in T_N \colon  \V u(y) \V >L\}} \V u(y) \V \dy$.
\end{enumerate}
\end{prop}
This statement follows directly from the proof of Theorem \ref{main} by simply truncating every component of $u$. Likewise, similar statements as in Propositions \ref{main2}, \ref{main3} and \ref{main3} follow for vector-valued differential forms.


\section{Applications} \label{secappl}
In the following, we consider a linear and homogeneous differential operator of first order, i.e. we are given $\A: C^{\infty}(\R^N,\R^d) \to C^{\infty}(\R^N,\R^l)$ of the form \begin{displaymath}
\A u= \sum_{k=1}^N A _k \partial_k u, 
\end{displaymath}
where $A_k : \R^d \to \R^l$ are linear maps. We call a continuous function $f: \R^d \to \R$ $\A$-quasiconvex if for all $\phi \in C^{\infty}(T_N,\R^d)$ with $\int_{T_N} \phi(y) \dy =0$ and $\A \phi =0$, and for all $x \in \R^d$ then the following version of Jensen's inequality \begin{equation} \label{Aqcdef}
f(x) \leq \int_{T_N} f(x + \phi(y)) \dy
\end{equation}
holds true.
\textsc{Fonseca} and \textsc{M\"uller} showed that \cite{FM}, if the constant rank condition seen below holds, then $\A$-quasiconvexity is a neccessary and sufficient condition for weak$*$ lower-semicontinuity of the functional $I: L^{\infty}(\Omega, \R^d) \to [0, \infty)$ defined by \begin{displaymath}
I(u) = \int_{\Omega} f(u(y)) \dy.
\end{displaymath}

\noindent Define the symbol  $\Aop \colon \R^N \back \{0\} \to \Lin(\R^d,\R^l)$ of the operator $\A$ by \begin{displaymath}
\Aop (\xi) = \sum_{k=1}^N \xi_k A_k.
\end{displaymath} 
The operator $\A$ is said to satisfy the \textbf{constant rank property} (c.f. \cite{Murat}) if for some fixed $r \in \{0,...,d\}$ and all $\xi \in \Sphere^{N-1} = \{\xi \in \R^N \colon \V \xi \V=1\}$ \begin{displaymath}
\dim (\ker \Aop(\xi)) =r.
\end{displaymath}
We call a homogeneous differential operator $\B:C^{\infty}(T_N,\R^m) \to C^{\infty}(T_N,\R^d)$, which is not neccessarily of order one, the potential of $\A$ if \begin{displaymath}
\psi \in C^{\infty}(T_N,\R^d) \cap \ker \A,~ \int_{T_N} \psi(y) \dy=0 ~\Longleftrightarrow \exists \phi \in C^{\infty}(T_N,\R^m) \text{ s. t.} \psi= \B \phi.
\end{displaymath}

Recently, \textsc{Rai\c{t}\u{a}} showed that $\A$ has such a potential if and only if $\A$ satisfies the constant rank property (\cite{Raita}). In the following, we always assume that $\A$ satisfies the constant rank property and that $\B$ is the potential of $\A$. 
\begin{defi}
We say that $\A$ satisfies the property (ZL) if it admits a statement of the type of Corollary \ref{main2}, i.e. for all sequences $u_n \in L^1(T_N,\R^d) \cap \ker \A$ such that there exists an $L>0$ with 
\begin{displaymath}
\int_{\{y \in T_N \colon \V u_n(y) \V >L \}} \V u_n(y) \V \dy \longrightarrow 0 \quad  \text{as } n \to \infty,
\end{displaymath}
there exists a $C=C(\A)$ and a sequence $v_n \in L^1(T_N,\R^d) \cap \ker \A$ such that \begin{enumerate} [i)]
\item $\norm v_n \norm_{L^{\infty}(T_N,\R^d)} \leq C_1 L$;
\item $\norm v_n - u_n \norm_{L^1(T_N,\R^d)} \to 0$ as $n \to \infty$.

\end{enumerate}
\end{defi}
Our goal now is to show that (ZL) implies further properties for the operator $\A$. We first look at a few examples.

\begin{ex} \label{ex:op} \begin{enumerate} [a)]
\item As shown by Zhang \cite{Zhang}, the operator $\A= \curl$ has the property (ZL). This is shown by using that its potential is the operator $\B= \nabla$. In fact, most of the applications here have been shown for $\B = \nabla$ relying on (ZL), but can be reformulated for $\A$ satisfying (ZL).
\item Let $W^k= (\R^N \otimes ... \otimes \R^N)_{\sym} \subset (\R^N)^k$. We may identify $u \in C^{\infty}(T_N,W^k)$ with $\tilde{u} \in C^{\infty}(T_N,(\R^N)^k)$ and define the operator \begin{displaymath}
\curl^{(k)} \colon C^{\infty}(T_N,W^k) \to C^{\infty}(T_N, (\R^N)^{k-1} \times \Lambda^2)
\end{displaymath} as taking the $\curl$ on the last component of $\tilde{u}$, i.e. for $I \in [N]^{k-1}$ \begin{displaymath}
(\curl^{(k)} u)_I = 1/2 \sum_{i,j \in \N} \partial_i \tilde{u}_{I j} - \partial_j \tilde{u}_{I i}  e_i \wedge e_j
\end{displaymath}
Note that this operator has the potential $\nabla^k:C^{\infty}(\R^N,\R) \to C^{\infty}(\R^N,W^k)$ (c.f. \cite{Meyers}). To the best of the author's knowledge the proof of the property (ZL) is in this setting not written down anywhere explicitly, but basically combining the works \cite{Acerbi,Francos,Stein,Zhang} yields the result (c.f. \cite{Schiffer2}).

\item In this work, it has been shown that the exterior derivative $d$ satisfies the property (ZL). The most prominent example is $\A=\divergence$.
\item The result is also true, if we consider matrix-valued functions instead (c.f. Proposition \ref{main4}).. For example, (ZL) also holds if we consider  $\divergence: C^{\infty}(\R^N,\R^{N \times M}) \to C^{\infty}(\R^N,\R^M)$, where \begin{displaymath}
\divergence _i u(x) = \sum_{j=1}^N \partial_j u_{ji} (x).
\end{displaymath}

\item Likewise, let $\A_1\colon C^{\infty}(T_N,\R^{d_1}) \to C^{\infty}(T_N,\R^{l_1})$ and $\A_1 \colon C^{\infty}(T_N,\R^{d_2}) \to C^{\infty}(T_N,\R^{l_2})$ be two differential operators satisfying (ZL). Then also the operator \begin{displaymath}
\A: C^{\infty}(T_N, \R^{d_1}\times \R^{d_2}) \to C^{\infty}(T_N,\R^{l_1}\times \R^{l_2})
\end{displaymath} defined componentwise for $u=(u_1,u_2)$ by \begin{displaymath}
\A (u_1,u_2) = (\A_1 u_1, \A_2 u_2)
\end{displaymath}
satisfies the property (ZL). The truncation is again done separately in the two components.
\end{enumerate}
\end{ex}
\begin{rem} This paper is concerned with $L^{\infty}$-truncation, i.e. we are in a situation with very low regularity. There also has been some progress when truncating divergence-free fields in $W^{1,\infty}$, using similar methods (c.f. \cite{Solenoidal} or \cite{Sebastian}). \end{rem}

An overview of the results one is able to prove using property (ZL) can be found in the lecture notes \cite[Sec. 4]{Mueller_var} and in the book \cite[Sec. 4,7]{Rindlerbook}, where they are formulated for the case of ($\curl$)-quasiconvexity. \medskip

\subsection{$\A$-quasiconvex hulls of compact sets} \label{Aqchulls}
For $f \in C(\R^d,\R)$ we can define the quasiconvex hull of $f$ by (c.f. \cite{FM,BFL}) \begin{equation} \label{fcthull}
\QA f(x) := \inf \left\{ \int_{T_N} f(x + \psi(y)) \dy \colon \psi \in C^{\infty}(T_N,\R^d) \cap \ker \A,~\int_{T_N} \psi =0 \right\}.
\end{equation} 
$\QA f$ is the largest $\A$-quasiconvex function below $f$ \cite{FM}.

In view of the separation theorem for convex sets in Banach spaces we define (c.f. \cite{CMO2,Sverak4,Sverak3}) the $\A$-quasiconvex hull  of a set $K \subset \R^d$ by  \begin{displaymath}
K^{\A qc}_{\infty} := \left \{ x \in \R^d \colon ~ \forall f \colon \R^d \to \R \text{ } \A \text{-quasiconvex with } f_{\V K} \leq 0 \text{ we have } f(x) \leq 0 \right \},
\end{displaymath}
and the $\A$-$p$-quasiconvex hull for $1 \leq p <\infty$ by 
\begin{align*}
K^{\A qc}_p:= &\left \{   x \in \R^d \colon ~ \forall f \colon \R^d \to \R \text{ } \A \text{-quasiconvex with } f_{\V K} \leq 0\right. \text{ and } \\~ &\quad \left.\V f(v) \V \leq C(1+ \V v \V^p)   \text{ we have } f(x) \leq 0 \right \}.
\end{align*}
The $\A$-$p$-quasiconvex hull for $1 \leq p <\infty$ can be alternatively defined via\begin{align*}
    K^{\A qc*}_p:= &\left \{   x \in \R^d \colon( \QA \dist^p(\cdot,K)) (x)=0 \right \}.
\end{align*}
If $K$ is compact, then $K^{\A qc}_p=K^{\A qc*}_p$.
Moreover, the spaces $K^{\A qc}_{p}$ are nested, i.e. $K^{\A qc}_{q} \subset K^{\A qc}_{q'}$ if $q \leq q'$.  In \cite{CMO2} it is shown that equality holds for $\A$ being the symmetric divergence of a matrix, $K$ compact  and $1 < q,q'< \infty$. The proof can be adapted for different $\A$, but uses the Fourier transform and is not suitable for the cases $p=1$ and $p=\infty$. Here, the property (ZL) comes into play.

For a compact set $K$ we define the set $K^{\A app}$ (c.f. \cite{Mueller_var}) as the set of all $x \in \R^d$ such that there exists a bounded sequence $u_n \in L^\infty(T_N,\R^d) \cap \ker \A$ with \begin{displaymath}
\dist(x +  u_n, K) \longrightarrow 0 \quad \text{in measure, as } n \to \infty.
\end{displaymath}

\begin{thm} \label{ZLappl} Suppose that $K$ is compact and $\A$ is an operator satisfying (ZL). Then \begin{equation} \label{set1}
K^{\A app} = K^{\A qc}_{\infty} = \left\{x \in \R^d \colon \QA (\dist(\cdot,K)) (x) = 0\right\}.
\end{equation}
\end{thm}

\begin{proof} 
We first prove $K^{\A app} \subset K^{\A qc}_{\infty}$. Let $x \in K^{\A app}$ and take an arbitrary $\A$-quasiconvex function $f: \R^d \to [0,\infty)$ with $f_{\V K} =0$. We claim that then $f(x)=0$.

Take a sequence $u_n$ from the definition of $K^{\A app}$. As $f$ is continuous and hence locally bounded, $f(u_n) \to 0$ in measure and $0 \leq f( u_n) \leq C$. Quasiconvexity and dominated convergence yield \begin{displaymath}
f(x) \leq \liminf_{n \to \infty} \int_{T_N} f(x +  u_n(y)) \dy = 0.
\end{displaymath}
$K^{\A qc}_{\infty} \subset \left\{x \in \R^d \colon \QA (\dist(\cdot,K)) (x) = 0\right\}$ is clear by definition, as $\QA (\dist(\cdot,K))$ is an admissible separating function.

The proof of the inclusion $\{ x \in \R^d \colon \QA (\dist(\cdot,K)) (x) = 0\} \subset  K^{\A app}$ uses (ZL). If $\QA (\dist(\cdot,K)) =0$, then there exists a sequence $\phi_n \in C^{\infty}(T_N,\R^d) \cap \ker \A$ with $\int_{T_N} \phi_n =0 $ such that \begin{displaymath}
0= \QA (\dist(\cdot,K))(x) = \lim_{n \to \infty} \int_{T_N} \dist(x +  \phi_n(y),K) \dy.
\end{displaymath}
As $K$ is compact, there exists $R>0$ such that $K \subset B(0,R)$. Moreover, as $x \in K_{\infty}^{\A qc}$, also $x \in B(0,R)$. This implies that \begin{displaymath}
\lim_{n \to \infty} \int_{T_N \cap \{\V  \phi_n \V \geq 6R\}} \V  \phi_n \V \dy =0.
\end{displaymath}
We may apply (ZL) and find a sequence $\psi_n \in L^{\infty}(T_N,\R^d)\cap \ker \A$ such that \begin{displaymath}
\norm  \phi_n -  \psi_n \norm_{L^1(T_N,\R^d)} \longrightarrow 0 \quad \text{as } n \to \infty
\end{displaymath}
and \begin{displaymath}
\norm \psi_n \norm _{L^{\infty}(T_N,\R^d)} \leq C R.
\end{displaymath}
Hence, $x \in K^{\A app}$.
 \end{proof}

\begin{rem} \label{remLp} Theorem \ref{ZLappl} shows that for all $1 \leq p < \infty$  
 \begin{displaymath}
 K^{\A app} = K^{\A qc}_{\infty} = \left\{x \in \R^d \colon \QA (\dist(\cdot,K)^p) (x) = 0\right\}
  = K^{\A qc}_{p}.
 \end{displaymath}
 This follows directly, as all the sets $K_p^{\A qc}$ are nested and, conversely, all the hulls of the distance functions are admissible $f$ in the definition of $K^{\A qc}_{\infty}$.
 \end{rem}
\begin{rem} Such a kind of theorem is not true for general unbounded closed sets $K$. As a counterexample one may consider $\A = \curl$ (i.e. usual quasiconvexity) and look at the set of conformal matrices $K= \{\lambda Q : \lambda \in \R^+, Q \in SO(n) \} \subset \R^{n \times n}$. If $n \geq 2$ is even(\cite{MSY}, there exists a quasiconvex function $F \colon  \R^{n \times n} \to \R$ with $F(x) = 0 \Leftrightarrow x \in K$ and \begin{displaymath}
0 \leq F(A) \leq C(1+ \V A \V^{n/2}). 
\end{displaymath}
On the other hand, let $n \geq 4$ be even and $F \colon \R^{n \times n} \to \R$ be a rank-one convex function with $F_{\V K} =0$ and for some $p < n/2$
\begin{displaymath}
0 \leq F(A) \leq C(1+ \V A \V^{p}).
\end{displaymath}
 Then $F=0$ by \cite{Yan}.

A reason for the nice behaviour of compact sets is that for such sets all distance functions are coercive, i.e. \begin{displaymath}
\dist(v,K)^p \geq \V v \V^p - C,
\end{displaymath}
which is obviously not true for non-compact sets. Coercivity of a function is often needed for relaxation results (c.f \cite{BFL}).
\end{rem}

\subsection{$\A$-$\infty$ Young measures}
We consider $\M(\R^d)$ the set of signed Radon measures with finite mass. Note that this is the dual space of $C_c(\R^d)$ with the dual pairing \begin{displaymath}
\li \mu ,f  \re = \int_{\R^d} f(y)~ \textup{d}\mu(y).
\end{displaymath}
For a measurable set $E \subset \R^N$ we call $\mu: E \to \M(\R^d)$ weak$*$ measurable if the map \begin{displaymath}
x \longmapsto  \li \mu(x),f \re 
\end{displaymath}
is measurable for all $f \in C_c(\R^d)$. Later, we may consider the space $L^{\infty}_w (E,\M(\R^d))$, which is the space of all weakly measurable maps such that $\spt \mu_x \subset B(0,R)$ for some $R>0$ and for a.e. $x \in T_N$. This space is equipped with the topology $\nu^n \weakstar \nu$ iff $\forall f \in C_0(\R^d)$ \begin{displaymath}
\li \nu^n_x , f \re \weakstar \li \nu_x , f \re \text{ in } L^{\infty}(E).
\end{displaymath}
\begin{rem}  \label{metrizable} The topology of $L^{\infty}_w(E,\M(\R^d))$ is metrisable on bounded sets: Note that $\nu_n$ supported on $B(0,R)$ converges to $\nu$  if and only if for all $f \in C(\bar{B}(0,R))$ and all $g \in L^1(E)$ \begin{displaymath}
\int_E \li \nu^n_x ,f \re g(x) \dx \longrightarrow \int_E \li \nu_x , f \re g(x) \dx.
\end{displaymath}
If $\nu^n$ is bounded, then this equation holds for all $f,g$ if and only if it holds for dense subsets of $C(\bar{B}(0,R))$ and $L^1(E)$. As these spaces are separable, we may consider a countable dense subset $(f_k,g_k)_{k \in \N}$ of $C(\bar{B}(0,R))\times L^1(E)$ and the metric \begin{displaymath}
d_k (\nu, \mu) = \left \V \int_E \li \nu_x - \mu_x ,f_k \re g_k(x) \dx \right \V,
\end{displaymath}
and then define the metric \begin{displaymath}
d(\nu, \mu) = \sum_{k \in \N} 2^{-k} \frac{d_k(\nu,\mu)}{1+ d_k(\nu,\mu)}.
\end{displaymath}
\end{rem}

 Let us now recall the Fundamental Theorem of Young measures(c.f. \cite{Ball}, \cite{compcomp}).

\begin{prop} [Fundamental Theorem of Young measures]  \label{FTOYM}
Let $E \subset \R^N$ be a measurable set of finite measure and $u_j: E \to \R^d$ a sequence of measurable functions. There exists a subsequence $u_{j_k}$ and a weak$*$ measurable map $\nu: E \to \M(\R^d)$ such that the following properties hold: \begin{enumerate} [i)]
\item $\nu_x \geq 0$ and $\norm \nu_x \norm_{\M(\R^d)} = \int_{\R^d} 1 ~\textup{d} \mu(x) \leq 1$;
\item $\forall f \in C_0(\R^d)$ define $\bar{f}(x) = \li \nu_x,f \re$.Then 
$f(u_{j_k}) \weakstar \bar{f} \text{ in } L^{\infty}(E)$;
\item If $K \subset \R^d$ is compact, then $\spt \nu_x \subset K$ if $\dist(u_{j_k},K) \to 0$ in measure;
\item  It holds \begin{equation} \label{1prime}
		\norm \nu_x \norm_{\M(\R^d)} = 1 \text{ for a.e. } x \in E \end{equation}
if and only if \begin{displaymath}
\lim_{M \to \infty} \sup_{k \in \N} \left \V \{ \V u_{j_k} \V \geq M\} \right \V =0;
\end{displaymath}
\item If \eqref{1prime} holds, then for all $A \subset E$ measurable and for all $f \in C(\R^d)$ such that $f(u_{j_k})$ is relatively weakly compact in $L^1(A)$, also \begin{displaymath}
f(u_{j_k}) \weakto \bar{f} \text{ in } L^1(A);
\end{displaymath}
\item If \eqref{1prime} holds, then (iii) holds with equivalence.
\end{enumerate}
\end{prop}

We call such a map $\nu: E \to \M(\R^d)$ the Young measure generated by the sequence $u_{j_k}$. One may show that every weak$*$ measureable map $E \to \M(\R^d)$ satisfying (i) is generated by some sequence $u_{j_k}$. 
\begin{rem}
If $u_k$ generates a Young measure $\nu$ and $v_k \to 0$ in measure (in particular, if $v_k \to 0$ in $L^1$), then the sequence $(u_k + v_k)$ still generates $\nu$.

If $u:T_N \to \R^d$ is a function, we may consider the oscillating sequence $u_n(x) := u(nx)$. This sequence generates the homogeneous (i.e. $\nu_x = \nu$ a.e.) Young measure $\nu$ defined by \begin{displaymath}
\li \nu , f \re = \int_{T_N} f(u_n(y)) \dy.
\end{displaymath}
\end{rem}
\begin{quest}
What happens if we impose further conditions on the sequence $u_{j_k}$, for instance $\A u_{j_k} =0$?
\end{quest}

For $1 \leq p < \infty$ we call a sequence $v_j \in L^p(\Omega,\R^d)$ $p$-equi-integrable if \begin{displaymath}
\lim \limits_{\varepsilon \to 0} \sup_{j \in \N} \sup_{E \subset \Omega \colon \V E \V < \varepsilon} \int_E \V v_j(y) \V ^p \dy = 0.
\end{displaymath}

\begin{defi} Let $1 \leq p \leq \infty$. We call a map $\nu \colon \Omega \to \R^d$ an $\A$-$p$-Young measure if there exists a $p$-equi-integrable sequence $\{v_j\} \subset L^p(\Omega, \R^d)$ (for $p = \infty$ a bounded sequence), such that $v_j$ generates $\nu$ and satisfies $\A v_j =0$.
\end{defi}

For $1 \leq p < \infty$ the set of $ \A $-$p$ Young measures was classified by \textsc{Fonseca} and \textsc{Müller} in \cite{FM} and for the special case $\A=\curl$ already in \cite{Pedregal}. 

\begin{prop}
Let $1 \leq p < \infty$. $\nu: T_N \to \M(\R^d)$ is an $\A$-$p$-Young measure if and only if \begin{enumerate} [i)]
\item $\exists v \in L^p(T_N,\R^d)$ such that $\A v =0$ and \begin{displaymath}
v(x) = \li \nu_x , \id \re = \int_{\R^d} y ~\textup{d}\nu_x(y) \text{ for a.e. } x \in T_N;
\end{displaymath}
\item $\int_{T_N} \int_{\R^d} \V z \V ^p ~\textup{d} \nu_x(z) \dx < \infty$;
\item for a.e. $x \in T_N$ and all continuous $g$ with $\V g(v) \V \leq C(1 + \V v \V^p)$ we have \begin{displaymath}
\li \nu_x, g \re \geq \QA g(\li \nu_x, \id \re).
\end{displaymath}
\end{enumerate}
\end{prop}

Recently, there has also been progress for so-called generalized Young measures ($p=1$ is a special case), c.f. \cite{RindlerKris,KirchKris,RindlerDP}.

Let us recall the result of \textsc{Kinderlehrer} and \textsc{Pedregal} for $W^{1,\infty}$-Gradient Young measures (c.f. \cite{Kinderlehrer}). 
\begin{prop}
A weak$*$ measurable map $\nu : \Omega \to \M (\R^{N \times m})$ is a $\curl$-$\infty$-Young measure if and only if $\nu_x \geq 0$ a.e. and there exists $K \subset \R^{N \times m}$ compact, $v \in L^{\infty}(\Omega,\R^{N \times m})$ such that \begin{enumerate} [a)]
\item $\spt \nu_x \subset K$ for a.e. $x \in \Omega$;
\item $ < \nu_x, \id> = v$ for a.e. $x \in \Omega$;
\item for a.e. $x \in \Omega$ and all continuous $g \colon \R^{N \times m} \to \R$ we have \begin{displaymath}
\li \nu_x, g \re \geq \mathcal{Q}_{\curl} g(\li \nu_x, \id \re).
\end{displaymath}
\end{enumerate}
\end{prop}
It is possible to state such a theorem in the general setting that $\A$ satisfies (ZL). The proofs in \cite{Kinderlehrer} mostly rely on this fact and the general case of operators satisfying (ZL) can be treated in the same fashion with few modifications. The details of the proofs can be found in the Appendix.

Let us first state the classification theorem for so called homogeneous $\A$-$\infty$-Young measures, i.e. $\A$-$\infty$-Young measures $ \nu \colon T_N \to \M(\R^d)$ with the following properties \begin{enumerate} [i)]
\item $\spt {\nu_x} \subset K$ for a.e. $x \in T_N$ where $K \subset \R^d$ is compact;
\item $\nu$ is a homogeneous Young measure, i.e. there exists $\nu_0 \in \M(\R^d)$ such that $\nu_x= \nu_0$ for a.e. $x \in T_N$.
\end{enumerate}

Define the set $\M^{\A qc} (K)$ by (c.f. \cite{Sverak2})
\begin{equation} \label{Maqc}
\M^{\A qc}(K) =\left\{ \nu \in \M(\R^d) \colon \nu \geq 0,~\spt \nu \subset K,~\li \nu,f \re \geq f( \li \nu, \id \re)~\forall f \colon \R^d \to \R \text{ } \A \text{-qc}\right\}.
\end{equation}
Denote by $H_{\A}(K)$ the set of homogeneous $\A$-$\infty$-Young measures supported on $K$. We are now able to formulate the classification of these measures (c.f.\cite[Theorem 5.1.]{Kinderlehrer}).
 \begin{thm}[Characterisation of homogeneous $\A$-$\infty$-Young measures] \label{hommain}
Let $\A$ satisfy the property (ZL) and $K$ be a compact set. Then \begin{displaymath}
H_{\A}(K) = \M ^{\A qc} (K).
\end{displaymath}
\end{thm}
Using this result, one may prove the Characterisation of $\A$-$\infty$-Young measures (c.f \cite[Theorem 6.1]{Kinderlehrer}).
\begin{thm}[Characterisation of $\A$-$\infty$-Young measures]\label{mainYM}
Suppose that $\A$ satisfies the property (ZL). A weak$*$ measurable map $\nu: T_N \to \M(\R^d)$ is an $\A$-$\infty$-Young measure if and only if $\nu_x \geq 0$ a.e. and there exists $K  \subset \R^d$ compact and $u \in L^{\infty}(T_N,\R^d) \cap \ker \A$ with \begin{enumerate} [i)]
\item $\spt \nu_x \subset K$ for a.e. $x \in T_N$.
\item $ \li \nu_x, \id \re = u$ for a.e. $x \in T_N$,
\item $\li \nu_x, f \re \geq f(\li \nu_x, \id \re)$ for a.e. $x \in T_N$ and  all continuous and $\A$-quasiconvex $f:\R^d \to \R$.
\end{enumerate}
\end{thm}
As mentioned, the proofs in the case $\A =\curl$ can be found in \cite{Kinderlehrer,Mueller_var,Rindlerbook} and in the general case in the Appendix. Let us shortly describe the strategy of the proofs. For Theorem \ref{hommain} one may prove that $H_{\A}(K)$ is weakly compact, that averages of (non-homogeneous) $\A$-$\infty$ Young measures are in $H_{\A}(K)$ and that the set $H_{A}^x(K)= \{ \nu \in H_{\A} \colon \li \nu, \id \re =x\}$ is weak$*$ closed and convex. The characterisation theorem then follows by using Hahn-Banach separation theorem and showing that any $\mu \in M^{\A qc}$ cannot be separated from $H_{\A}(K)$, i.e. for all $f \in C(K)$ and for all $\mu \in M^{\A qc}(K)$ with $\li \mu, \id \re=0 $ \begin{displaymath}
\li \nu, f \re \geq 0 \text{ for all } \nu \in H_{\A}^0(K) \Rightarrow \li \mu, f \re \geq 0.
\end{displaymath}
Theorem \ref{mainYM} then can be shown using Proposition \ref{hommain} and a localisation argument.

\textbf{Acknowledgements:} The author would like to thank Stefan M\"uller for introducing him to the topic and for helpful discussions. The author has been supported by the Deutsche Forschungsgemeinschaft (DFG, German Research Foundation) through the graduate school BIGS of the Hausdorff Center for Mathematics (GZ EXC 59 and 2047/1, Projekt-ID 390685813).

\begin{appendix}
\section{Proofs of Theorem \ref{hommain} and Theorem \ref{mainYM}}
We want to prove the characterisation Theorem \ref{hommain}. For this result we shall prove some auxiliary lemmas first. We mainly follow the proofs in the $\curl$-free setting in \cite{Kinderlehrer,Rindlerbook,Mueller_var}.
\begin{lemma} (Properties of $H_{\A}(K)$) \label{prop1}
\begin{enumerate} [i)]
\item If $\nu \in H_{\A}(K)$ with $\li \nu, \id \re =0$, then there exists a sequence $u_j \in L^{\infty}(T_N,\R^d)$ such that $\A u_j   =0$, $u_j$ generates $\nu$ and $\norm u_j \norm_{L^{\infty}(T_N,\R^d)} \leq C \sup_{z \in K} \V z \V = C \V K \V_{\infty}$.
\item $H_{\A}(K)$ is weakly$*$ compact in $\M (\R^d)$.
\end{enumerate}
\end{lemma}

\begin{proof}
i) follows from the definition of $H_{\A}(K)$. The uniform bound on the $L^{\infty}$ norm of $u_j$ can be guaranteed by (ZL) and vi) in Theorem \ref{FTOYM}. 

For the weak$*$ compactness note that $H_{\A}(K)$ is contained in the weak$*$ compact set $\P(K)$ of probability measures on $K$. As the weak$*$ topology is metrisable on $\P(K)$ it suffices to show that $H_{\A}(K)$ is sequentially closed. Hence, we consider a sequence $\nu_k \subset H_{\A}(K)$ with $\nu_k \weakstar \nu$ and show that $\nu \in H_{\A}(K)$

Due to the definition of Young measures, we may find sequences $u_{j,k} \in L^{\infty}(T_N,\R^d) \cap \ker \A$ such that $u_{j,k}$ generates $\nu_k$ for $j \to \infty$.
Recall that the topology of generating Young measures is metrisable on bounded set of $L^{\infty}(T_N,\R^d)$ (c.f. Remark \ref{metrizable}). We may find a subsequence $u_{j_k,k}$ which generates $\nu$. As we know that $\norm u_{j_k,k} \norm_{L^{\infty}} \leq C \V K \V_{\infty}$, $\nu \in H_{\A}(K)$ and hence $H_{\A}(K)$ is closed.
\end{proof}
\begin{lemma} \label{average} 
Let $\nu$ be an $\A$-$\infty$-Young measure generated by a bounded sequence $u_k \in L^{\infty}(T_N,\R^d) \cap \ker \A$. Then the measure $\bar{\nu}$ defined via duality for all $f \in C_0(\R^d)$ by \begin{displaymath}
\li \bar{\nu},f \re = \int_{T_N} \li \nu_x ,f \re \dx
\end{displaymath}
is in $H_{\A}(K)$.
\end{lemma}
\begin{proof}
For $n \in \N$ define $u_k^n \in L^{\infty}(T_N,\R^d) \cap \ker \A$ by $u_k^n(x)=u_k(nx)$. Then for all $f \in C_0(\R^d)$ \begin{displaymath}
f(u_k^n) \weakstar \int_{T_N} f(u_k) \text{ in } L^{\infty}(T_N,\R^d) \quad \text{as } n \to \infty .
\end{displaymath}
Note that by Theorem \ref{FTOYM} ii) we also have \begin{displaymath}
\int_{T_N} f(u_k(x)) \dx \longrightarrow \int_{T_N} \li \nu_x, f \re \dx \quad \text{ as } k \to\infty.
\end{displaymath}
Due to metrisability on bounded sets (Remark \ref{metrizable}), we can find a subsequence $u_k^{k(n)}$ in $L^{\infty}(T_N,\R^d)$ such that \begin{displaymath}
f(u_k^{n(k)}) \weakstar \int_{T_N} \li \nu_x, f \re \dx \quad \text{as } k \to \infty.
\end{displaymath}
Thus, $\bar{\nu} \in H_{\A}(K)$.
\end{proof}
\begin{lemma} 
Define the set $H_{\A}^x(K) := \{ \nu \in H_{\A} \colon \li \nu, \id \re = x\}$. Then $H_{\A}^x(K)$ is weak$*$ closed and convex. 
\end{lemma}

\begin{proof}
Weak$*$ closedness is clear by Lemma \ref{prop1}. It suffices to show that for any $\nu_1$, $\nu_2 \in H_{\A}^x(K)$ and $\lambda \in [0,1]$ it holds $\lambda \nu_1 + (1- \lambda) \nu_2 \in H_{\A}^x(K)$. Due to Lemma \ref{prop1} there exist bounded sequences $u_j$, $v_j \in L^{\infty}(T_N,\R^d) \cap \ker \A$ generating $\nu_1$ and $\nu_2$, respectively.

Note that $u_j$ generates $\nu_1$ iff for all $f \in C_c(\R^d)$ \begin{displaymath}
f(u_j) \weakstar \int_{\R^d} f~ \textup{d}\nu_1 \text{ as } j \to \infty \text{ in } L^{\infty}(T_N,\R^d).
\end{displaymath}
Now consider the sequence $u_j^n(\cdot)$ = $u_j(n\cdot)$ for $n \in \N$. We know that for all $f \in C_c(\R^d)$ \begin{displaymath}
f(u_j^n) \weakstar \int_{T_N} f(u_j(y)) \dy \quad  \text{ as } n \to \infty  \text{ in } L^{\infty}(T_N,\R^d).
\end{displaymath}

As $\nu_1$ is a homogeneous measure, we also know that \begin{displaymath}
f(u_j) \weakstar \int_{\R^d} f~ \textup{d}\nu_1 \quad \text{as } j \to \infty \text{ in } L^{\infty}(T_N,\R^d).
\end{displaymath}
By a diagonalisation argument (which works since $L^{\infty}_w(T_N,\M(\R^d))$ is metrisable, Remark \ref{metrizable}), we may pick a subsequence $u_j^{n(j)}$ with the following properties: \begin{enumerate} [a)]
\item $u_j^{n(j)}$ generates the homogeneous Young measure $\nu_1$.
\item $k(j) \to \infty$ as $j \to \infty$.
\end{enumerate}

We now use that we have a potential operator $\B$ of degree $k$ (c.f. \cite{Raita}). Hence, we may find $U_j \in W^{k,2}(T_N,\R^d)$ satisfying \begin{enumerate} [i)]
\item $\B U_j = u_j$;
\item $\norm U_j \norm _{W^{k,2}(T_N,\R^m)} \leq C \norm u_j \norm_{L^2(T_N,\R^d)} \leq C \norm u_j \norm_{L^{\infty}(T_N,\R^d)}$.
\end{enumerate}
By scaling, we get potentials $U_j^{n(j)} = n(j)^{-k} U_j(n(j)x)$ such that $\B U_j^{n(j)} = u_j^{n(j)}$.

Consider smooth cut-off functions $\phi_n \in C_c^{\infty}((0,\lambda)\times(0,1)^{N-1}, [0,1])$  such that $\phi_n \to 1$ pointwise and in $L^1((0,\lambda)\times(0,1)^{N-1},\R)$, $\phi_n (x) =1 $ if $\dist(x,\partial ((0,\lambda)\times(0,1)^{N-1}) > \frac{1}{n}$, and for all $l \in \N$ \begin{displaymath}
\norm \nabla^l \phi_n \norm_{L^{\infty}} \leq C_l n^l.
\end{displaymath}

We define $\tilde{U}_j^{n(j)} = \phi_{n(j)} \cdot U_j^{n(j)}$ and $\tilde{u}_j = \B \tilde{U}_j^{n(j)}$. Note that \begin{enumerate} [1)]
\item $\A \tilde{u}_j = 0$ as it is element of the image of $\B$;
\item $\tilde{u}_j$ is compactly supported in $(0,\lambda)\times(0,1)^{N-1}$;
\item $\norm \tilde{u}_j - u_j^{k(j)} \norm_{L^1} \to 0$ as $j \to \infty$.
\end{enumerate}
As $u_j^{k(j)}$ is in $L^{\infty}$ with uniform bound $C \V K \V_{\infty}$ (c.f. Lemma \ref{prop1}), property 3) implies \begin{equation} \label{Zhangbound1}
\int_{\{y \in (0,\lambda)\times(0,1)^{N-1} \colon \V \tilde{u}_j(y) \V \geq C \V K  \V_{\infty}\}} \V \tilde{u}_j(y) \V \dy \longrightarrow 0 \quad \text{as } j \to \infty.
\end{equation}
We may construct a similar sequence $\tilde{v}_j$ satisfying $\tilde{v}_j =0$, $\tilde{v}_j$ is compactly supported  in $(\lambda,1)\times (0,1)^{N-1}$ and also satisfies an estimate of type \ref{Zhangbound1}. Due to property 3), $\tilde{u}_j$ still generates the homogeneous Young measure $\nu_1$. Hence, \begin{displaymath}
w_j := \left\{ \begin{array}{ll} \tilde{u}_j & \text{on } (0,\lambda) \times (0,1)^{N-1}, \\ 
								\tilde{v}_j & \text{on } (\lambda,1) \times (0,1)^{N-1} \end{array} \right.
\end{displaymath}
satisfies $\A w_j =0$, $w_j$ is compactly supported on the unit cube, and due to \eqref{Zhangbound1}, \begin{displaymath}
\int_{\{\V w_j \V \geq C \V K  \V_{\infty}\}} \V w_j(y) \V \dy \longrightarrow 0 \quad \text{as } j \to \infty.
\end{displaymath}

Moreover, $w_j$ generates the Young measure $\nu$ with $\nu_x = \nu_1$ on $(0,\lambda) \times [0,1]^{N-1}$ and $\nu_x =\nu_2 $ on $(\lambda,1) \times [0,1]^{N-1}$. 
Again, consider the oscillating version $w_j^n(x) = w_j (nx)$, which generates $\int_{T_N} f(w_j) \dy$ as a Young measure. By the same argument as before taking a suitable subsequence $w_j^{n(j)}$, this subsequence generates $\lambda \nu_1+ (1-\lambda) \nu_2$. Note that $w_j^{n(j)}$ is not in $L^{\infty}$, but it is almost in $L^{\infty}$, i.e. \begin{displaymath}
\int_{\{ y \in T_N \colon \V w_j^{n(j)} \V \geq C \V K \V_{\infty} \}} \V w_j^{n(j)} \V \dy \longrightarrow 0 \quad \text {as } j \to \infty.
\end{displaymath}
Hence, by property (ZL), there exists a $\widetilde{w}_j$ uniformly bounded in $L^{\infty}(T_N,\R^d) \cap \ker \A$ and with $\norm w_j - w_j^{n(j)} \norm_{L^1(T_N,\R^d)} \to 0$ as $j \to \infty$. Thus, $\widetilde{w}_j$ still generates $\lambda \nu_1 + (1- \lambda) \nu_2$ and is bounded in $L^{\infty}$. We conclude that $\lambda \nu_1 + (1- \lambda) \nu_2 \in H_{\A}^x(K)$ and thus this set is convex.

\end{proof}
 
\begin{rem} \label{boundary}
Using the construction we made in the proof (one can choose $\lambda=1$), we may also prove the following characterisation of $\A$-$\infty$-Young measures: $\nu$ is a homogeneous $\A$-$\infty$-Young measure if there exists a sequence $u_n \in L^{1}(\R^N,\R^d) \cap \ker \A$ with $\spt u_n \subset Q = [0,1]^N$ and \begin{displaymath}
\int_{\{\V u_n \V \geq L\}} \V u_n \V \longrightarrow 0 \quad \text{as } n \to \infty.
\end{displaymath}
By convolution we may even assume that $u_n \in C_c^{\infty}((0,1)^N,\R^d) \cap \ker \A$.
\end{rem}

We are now ready to prove Theorem \ref{hommain}.
\begin{proof}[Proof of Theorem \ref{hommain}:] We have that $H_{\A}(K) \subset \M^{\A qc}$ due to the fundamental theorem of Young measures: $\nu \geq 0$ and $\spt \nu \subset K$ are clear by i) and iii) of Theorem \ref{FTOYM}.
The corresponding inequality follows by $\A$-quasiconvexity, i.e. if $u_n \in L^{\infty}(T_N,\R^d) \cap \ker \A$ generates the Young measure $\nu$, then 
\begin{displaymath}
\li \nu, f \re = \lim_{n \to \infty} \int_{T_N}  f (u_n(y)) \dy 
		\geq \liminf_{n \to \infty} f\left(\int_{T_N} u_n(y) \dy \right) 
		=  f(\li \nu, \id \re).
\end{displaymath}
To prove $M^{\A qc}(K) \subset H_{\A}(K)$, w.l.o.g. consider a measure such that $\li \nu, \id \re=0$. We just proved that $H_{\A}^0(K)$ is weak$*$ closed and convex. Remember that $C(K)$ is the dual space of the space of signed Radon measures $\M(K)$ with the weak$*$ topology (see e.g. \cite{Rudin}). Hence, by Hahn-Banach separation theorem, it suffices to show that for all $f \in  C(K)$ and all $\mu \in M^{\A qc}(K)$ with $\li \mu,\id \re=0$\begin{displaymath}
\li \nu, f \re \geq 0 \text{ for all } \nu \in H_{\A}^0(K)  \Rightarrow \li \mu, f \re \geq 0.
\end{displaymath}
To this end, fix some $f \in C(K)$, consider a continuous extension to $C_0(\R^d)$ and let \begin{displaymath}
f_k(x) = f(x) + k \dist^2(x,K).
\end{displaymath}
We claim that  \begin{equation}  \label{homogclaim}
\lim_{k \to \infty} \QA f_k(0) \geq 0.
\end{equation}
If we show \eqref{homogclaim}, $\mu$ satisfies \begin{displaymath}
\li \mu,f \re = \li \mu, f_k \re \geq \li \mu , \QA f_k \re  \geq \QA f_k(0),
\end{displaymath}
finishing the proof. For the identity $\li \mu,f \re = \li \mu, f_k \re $ recall that $\mu $ is supported in $K$ and $\dist^2(x,K) = 0$ for $x \in K$.

Hence, suppose that \eqref{homogclaim} is wrong. As $f_k$ is strictly increasing, there exists $\delta >0$ such that 
\begin{displaymath}
\QA f_k(0) \leq - 2 \delta, \quad k \in \N.
\end{displaymath}
Using the definition of the $\A$-quasiconvex hull \eqref{fcthull}, we get $u_k \in L^{\infty}(T_N,\R^d) \cap \ker \A$ with $\int_{T_N} u_k(y) \dy =0$ and \begin{equation} \label{contradict1}
\int_{T_N} f_k (u_k(y)) \dy \leq -\delta.
\end{equation}
We may assume that $u_k \weakto 0$ in $L^2(T_N,\R^d)$ and also that $\dist(u_k,K) \to 0$ in $L^1(T_N)$. By property (ZL), there exists a sequence $v_k \in \ker \A$ bounded in $L^{\infty}(T_N,\R^d)$ with $\norm u_k -v_k \norm_{L^1} \to 0$. $v_k$ generates (up to taking subsequences) a Young measure $\nu$ with $\spt \nu_x \subset K$.

Then for fixed $j \in \N$, using Lemma \ref{average} and that $\bar{\nu} \in H_{\A}(K) \subset M^{\A qc}(K)$, \begin{displaymath}
\liminf_{k \to \infty} \int_{T_N} f_j(u_k(y)) \dy \geq \liminf_{k \to \infty} \int_{T_N} f_j(v_k(y)) \dy 
 = \int_{T_N} \int_{\R^d} f_j ~\textup{d} \nu_x \dx = \li \bar{\nu}, f \re \geq 0.
\end{displaymath}
But this is a contradiction to \eqref{contradict1}, as $f_k \geq f_j$ if $k \geq j$.
\end{proof}
Using the result of Theorem \ref{hommain}, we are now able to prove Theorem \ref{mainYM}.
\begin{proof}[Proof of Theorem \ref{mainYM}:]
We first establish the neccessity of i) - iii). If $\nu$ is an $\A$-$\infty$-Young measure then there exists $u_n \in L^{\infty}(T_N,\R^d) \cap \ker \A$ that generates $\nu$. As $u_n$ is bounded in $L^{\infty}$, $\norm u_n \norm_{L^{\infty}} \leq R$, then $\spt \nu_x \subset \bar{B}(0,R)$.

Moreover, $u_n \weakstar \li \nu_x, \id \re=: u$, hence the function $u$ satisfies $\A u=0$. iii) follows from the weak$*$ lower semicontinuity result from \cite{FM}. Note that  for $f \in C(\R^d)$ $\A$-quasiconvex the value of $f(u_n)$ only depends on $f: B(0,R) \to \R$, hence we may consider a dense countable subset of $C(B(0,R)) $. Then for a.e. $x \in T_N$ and all $f$ in this countable subset we may use Lebesgue point theorem and see that \begin{align*}
\li \nu , f \re &= \lim_{r \to 0} \fint_{B_r(x)} f(y) ~ \textup{d} \nu_y \\
			& = \lim_{r \to 0} \lim_{n \to \infty} \fint_{B_r(x)} f(u_n(y)) \dy \\
			&\geq \lim_{r \to 0} \fint_{B_r(x)} f(u(y)) \dy \\
			&= f( \li \nu_x, \id \re).
\end{align*}

 For sufficiency consider such a $\nu \geq 0$ satisfying i) - iii). Suppose first that $\li \nu_x, \id \re =0$ for a.e. $x \in \R^d$.

We divide $T_N$ into $n^N$ equisized subcubes $Q^n_i = q_i +[0,1/n)^N$, $i \in [n^N]$ where $q_i \in 1/n \Z^N$. Define $\nu^n_i$ by duality as \begin{equation} \label{nuin}
\li \nu_i^n, f \re = \fint_{Q^n_i} \li \nu_x, f \re \dx
\end{equation}
and $\nu^n: T_N \to \mathcal{M}(\R^d)$ by \begin{displaymath}
(\nu^n)_x = \nu^n _i \text{ if } x \in Q^n_i.
\end{displaymath}
Note that $\nu^n_i$ is bounded and for all $f \in C_c(\R^d)$ \begin{displaymath}
\li \nu^n , f \re \longrightarrow \li \nu, f \re \text{ in } L^1(\R^d).
\end{displaymath}
Due to boundedness we also know that \begin{displaymath}
\nu^n \weakstar \nu \text{ in } L^{\infty}_w(T_N,\M(\R^d)).
\end{displaymath}
We only need to show that $\nu^n$ is an $\A$-$\infty$ Young measure, as then we may take a diagonal sequence to get the Young measure $\nu$ (recall that the topology was metrisable on bounded sets, Remark \ref{metrizable}).

Fix some $n \in \N$. Note that for any $i \in [n^N]$, by Theorem \ref{hommain}, the measure $\nu^n_i$ is a homogeneous $\A$-$\infty$-Young measure. Indeed, \begin{enumerate} [i)]
    \item $\nu_i^n \geq 0$,
    \item $\spt(\nu_i^n) \subset K$, as for every $x \in Q^n_i$ $\spt \nu_x \subset K$,
    \item  for every $\A$-quasiconvex $f\in C(\R^d)$ we have \begin{align*}
        \li \nu_i^n , f\re = \fint_{Q^n_i} \li \nu_x,f \re 
                            \geq \fint_{Q^n_i} f(\li \nu_x, \id \re)
                            = f(0) = f(\li \nu_i^n, \id \re). 
    \end{align*}
\end{enumerate} Hence, by Remark \ref{boundary} there exists a sequence $u^k_i$ generating $\nu^n_i$ with $u^k_i \in C_c^{\infty}(Q,\R^d) \cap \ker \A$ and \begin{displaymath}
\int_{ \V u^k_i \V \geq C R} \V u_k^i \V \leq 1/k.
\end{displaymath}
By rescaling, we may find $\tilde{u}^k_i \in C_c^{\infty}(Q^n_i) \cap \ker \A$ with $\int_{\V \tilde{u}^k_i\V \geq C  R} \V \tilde{u}_k^i \V \leq n^{-N}/k$.
Thus, we may find $v_k \in L^1(T_N,\R^d) \cap \ker \A$ with \begin{displaymath}
\int_{\V v_k \V \geq CL} \V v_k \V \leq 1/k,
\end{displaymath}
by defining $v_k = \tilde{u}^k_i$ on $Q^n_i$.  $v_k$ generates $\nu^N$ by construction.
By (ZL), we may find $\tilde{v}_k \in L^{\infty}(T_N,\R^d) \cap \ker \A$ bounded with $\norm \tilde{v}_k - v_k \norm_{L^1(T_N,\R^d)} \to 0$, which therefore still generates $\nu^N$. Thus, $\nu^N$ is an $\A$-$\infty$-Young measure and then also $\nu$.

For $\li \nu_x, \id \re = u$ consider the translated measure $\tilde{\nu}$ defined by \begin{displaymath}
\li \tilde{\nu}, f \re = \int_{\R^d} f(y + u_n(x)) ~\textup    {d} \nu_x.
\end{displaymath}
Note that $\li \tilde{\nu}_x, \id \re = 0$ and $\tilde{\nu}$ still satisfies i)-iii). Then we may find a bounded sequence $w_k \in L^{\infty}(T_N,\R^d) \cap \ker \A$ generating $\tilde{u}$. The functions $w_k + u \in L^{\infty}(T_N,\R^d) \cap \ker \A$ then generate $\nu$.
\end{proof}

\end{appendix}
\bibliographystyle{alpha}
\bibliography{diffform_2}

\end{document}